\documentclass[a4paper,12pt]{article}
\usepackage[centertags]{amsmath} 
\usepackage{graphicx}
\usepackage{amssymb}
\usepackage{indentfirst}
\usepackage{amsmath}
\usepackage{amsthm}  
\usepackage{color}
\usepackage{subcaption}
\usepackage{float}
\usepackage{wrapfig}
\usepackage{enumitem}
                                                                                                  
           \theoremstyle{thm}
            \newtheorem{theorem}{Theorem}[section]                                               
            \newtheorem{lemma}[theorem]{Lemma}
            \newtheorem{corollary}[theorem] {Corollary}
            \newtheorem{proposition}[theorem]{Proposition}
            
            \theoremstyle{definition}
            \newtheorem{definition}[theorem]{Definition}
             
             \newtheorem{remark}[theorem]{Remark}
\title{A dynamical characterization of acylindrically hyperbolic groups}
\author{Bin\quad Sun\\ \small{Department of Mathematics, Vanderbilt University}\\\small{Email: bin.sun@vanderbilt.edu}}   
\date{}                     
\begin{document}
\maketitle{}
\begin{center}
\textbf{Abstract}
\end{center}

We give a dynamical characterization of acylindrically hyperbolic groups. As an application, we prove that non-elementary convergence groups are acylindrically hyperbolic.

\section{Introduction}\label{sec.intro}

The notion of an acylindrically hyperbolic group was introduced by Osin \cite{acylindrically hyperbolic group}. A group is called {\it acylindrically hyperbolic} if it admits a non-elementary acylindrical action on a Gromov hyperbolic space (for details, see Section \ref{subsec.ah}). Non-elementary hyperbolic and non-elementary relatively hyperbolic groups are acylindrically hyperbolic. Other examples include all but finitely many mapping class groups of punctured closed surfaces, outer automorphism groups of non-abelian free groups, many of the fundamental groups of graphs of groups, groups of deficiency at least two, etc (see Osin \cite{O} for details and other examples).

Not only do acylindrically hyperbolic groups form a rich class, but they also enjoy various nice algebraic, geometric and analytic properties. For example, every acylindrically hyperbolic group $G$ has non-trivial $H^{2}_{b}(G,\ell^{2}(G))$, which allows one to apply the Monod-Shalom rigidity theory \cite{MS} for measure preserving actions. Using methods from Dahmani-Guirardel-Osin \cite{DGO}, one can also find hyperbolically embedded subgroups in acylindrically hyperbolic groups and then use group theoretic Dehn surgery to prove various algebraic results (e.g., SQ-universality). Yet there is also a version of the small cancellation theory for acylindrically hyperbolic groups (see Hull \cite{Hull}). For a brief survey on those topics we refer to Osin \cite{acylindrically hyperbolic group,O}.

The work of Bowditch \cite{construction}, Freden \cite{Freden}, and Tukia \cite{property} provides a dynamical characterization of non-elementary hyperbolic groups by means of the notion of convergence groups. An action of a group $G$ on a metrizable topological space $M$ is called a \textit{convergence action} (or $G$ is called a \textit{convergence group} acting on $M$) if the induced diagonal action of $G$ on the space of distinct triples
$$\Theta_{3}(M)=\{(x_{1},x_{2},x_{3})\in M\mid x_{1}\neq x_{2},\ x_{2}\neq x_{3},\ x_{1}\neq x_{3}\}$$
is properly discontinuous. Convergence groups were introduced by Gehring-Martin \cite{introduction} in order to capture the dynamical properties of Kleinian groups acting on the ideal spheres of real hyperbolic spaces. Although the original paper refers only to actions on spheres, the notion of convergence groups can be generalized to general compact metrizable topological spaces or even compact Hausdorff spaces. Bowditch \cite{construction,convergence group} and Tukia \cite{property} proved that non-elementary hyperbolic groups are precisely uniform convergence groups acting on perfect compact metrizable topological spaces. Later, a characterization of relatively hyperbolic groups was given by Yaman \cite{relative}.

Inspired by the result of Bowditch and Tukia, we introduce Condition (C) for group actions on topological spaces and use it to characterize acylindrically hyperbolic groups. 

\begin{definition}\label{3.5}
Given a group $G$ acting by homeomorphisms on a topological space $M$ which has at least $3$ points, we consider the following condition (see Figure \ref{fig xiao}):

\begin{enumerate}[leftmargin=2em]
\item[(C)] For every pair of distinct points $u,v\in\Delta=\{(x,x)\mid x\in M\}$, there exist open sets $U,V$, of the product topological space $M^2$, containing $u,v$ respectively, such that for every pair of distinct points $a,b\in M^2\backslash \Delta$, there exist open sets $A,B$, of the product topological space $M^2$ ($A,B$ are permitted to intersect $\Delta$), containing $a,b$ respectively, with
$$|\{g\in G\mid gA\cap U\neq\emptyset,gB\cap V\neq\emptyset\}|<\infty.$$
\end{enumerate}
\end{definition}

\begin{theorem}\label{1.1}
A non-virtually-cyclic group $G$ is acylindrically hyperbolic if and only if $G$ admits an action on some completely Hausdorff topological space $M$ satisfying (C) with an element $g\in G$ having north-south dynamics on $M$.
\end{theorem}

Recall that a topological space $M$ is called \textit{completely Hausdorff} if for any two distinct points $u,v\in M$, there exist open sets $U,V$ containing $u,v$ respectively, such that $\overline{U}\cap \overline{V}=\emptyset$. Also recall that a element $g\in G$ is said to have north-south dynamics on $M$ if $g$ fixes exactly two points $x\neq y$ of $M$ and ``translates'' everything outside of $x$ towards $y$ (see Definition \ref{3.4} for details).

\begin{figure}
\begin{center}
\resizebox{0.5\linewidth}{!}{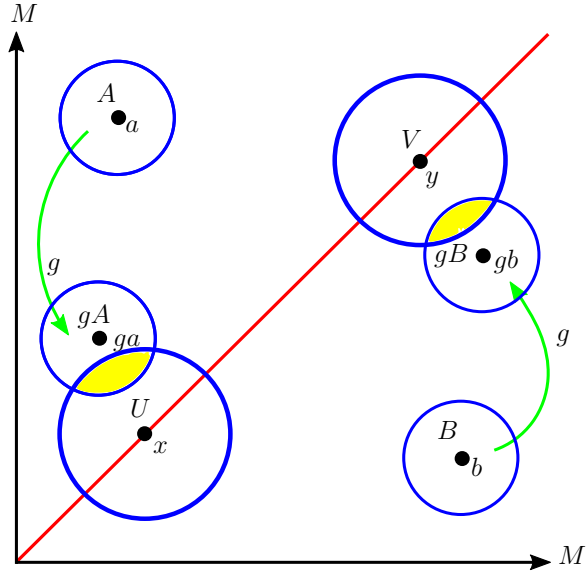}
~~~\caption{The (C) condition}\label{fig:gcg}
\label{fig xiao}
\end{center}
\end{figure}

It was established earlier that non-elementary convergence groups are non-virtually-cyclic and contain elements with north-south dynamics. Thus, by proving that every convergence action satisfies (C), we obtain the following corollary.

\begin{corollary}\label{1.2}
Non-elementary convergence groups are acylindrically hyperbolic.
\end{corollary}

Karlsson \cite[Proposition 6]{K} proved that if $G$ is a finitely generated group whose Floyd boundary $\partial_{F}G$ has cardinality at least 3, then $G$ acts on $\partial_{F}G$ by a non-elementary convergence action. Thus, as a further application of Theorem \ref{1.1}, we recover the following result.

\begin{corollary}
\upshape (Yang \cite[Corollary 1]{floyd}) \itshape Every finitely generated group with Floyd boundary of cardinality at least $3$ is acylindrically hyperbolic.
\end{corollary}

The converse of Corollary \ref{1.2} is not true, i.e., there exists an acylindrically hyperbolic group such that every convergence action of this group is elementary. In Section \ref{sec.p}, we are going to prove that mapping class groups of closed orientable surfaces of genus at least $2$ and non-cyclic directly indecomposible right-angled Artin groups corresponding to connected graphs are examples of this kind.

For countable groups, applying a result of Balasubramanya \cite{sahana}, we show

\begin{theorem}\label{7.6}
A non-virtually-cyclic countable group $G$ is acylindrically hyperbolic if and only if $G$ admits an action on the Baire space satisfying (C) and contains an element with north-south dynamics. 
\end{theorem}

Recall that the \textit{Baire space} is the Cartesian product $\mathbb{N}^{\mathbb{N}}$ with the Tychonoff topology. Theorem \ref{7.6} implies that acylindrical hyperbolicity of countable groups can be characterized by their actions on a particular space, the Baire space.

This paper is organized as follows. In Section \ref{sec.h}--\ref{cg}, we survey some basic information about Gromov hyperbolic spaces, acylindrically hyperbolic groups, and convergence groups. We introduce the notion of Condition (C) in Section \ref{sec.gcg}. In Section \ref{sec.a}, we survey a construction due to Bowditch \cite{construction}. The proof of Theorem \ref{1.1} is presented in Section \ref{sec.p} and separated into two parts. We first use geometric properties of Gromov hyperbolic spaces to prove that every acylindrically hyperbolic group is non-virtually-cyclic and admits an action satisfying (C) on a completely Hausdorff space with an element having north-south dynamics. The other direction of Theorem \ref{1.1} is proved by using the construction of Bowditch. We also prove Theorem \ref{7.6} and discuss Corollary \ref{1.2} and its converse in Section \ref{sec.p}.

\noindent \textbf{Acknowledgment:}
I would like to thank my  supervisor, Denis Osin, for introducing me to the subject, for explaining his view on this topic, and for his proofreading for this paper. This paper would not have been written without his help. I would also like to thank the referee, who helped me make this article more precise and clear.


\section{Gromov Hyperbolic spaces}\label{sec.h}

\subsection{Definition}

We start by recalling the well-known concept of a Gromov hyperbolic space. Suppose that $(S,d)$ is a geodesic metric space with underlying space $S$ and metric $d$. Let $\Delta$ be a geodesic triangle consisting of three geodesic segments $\gamma_1,\gamma_2,\gamma_3$. For a number $\delta\geqslant 0$, $\Delta$ is called $\delta$\textit{-slim} if the distance between every point of $\gamma_i$ and the union $\gamma_j\cup\gamma_k$ is less than $\delta$, where $i,j,k\in\{1,2,3\},i\neq j,j\neq k,k\neq i$.

We say that $(S,d)$ is a $\delta$\textit{-hyperbolic space} if geodesic triangles in $S$ are all $\delta$-slim. $(S,d)$ is called a \textit{Gromov hyperbolic space} if it is $\delta$-hyperbolic for some $\delta\geqslant 0$. Gromov hyperbolic spaces generalize notions such as simplicial trees and complete simply connected Riemannian manifolds with constant negative sectional curvature while preserving most of the interesing properties (see Bridson-Haefliger \cite{BH}, V\"ais\"al\"a \cite{hyperbolic space}).

\noindent
\textbf{Some notations:}
When one refers to a metric space $(S,d)$, usually there is no ambiguity of the metric $d$ once the underlying space $S$ is clarify. Thus, we will omit the metric and just use a single letter $S$ to indicate a metric space whenever there is no ambiguity of the metric. Also, for every $x\in S$ and $r>0$, we will use $B_S(x,r)$ to denote the open ball in $S$ with $x$ as its center and $r$ as its radius.

\begin{remark}
In literature, properness is often part of the definition of a Gromov hyperbolic space. However, in this article, we do not assume that a Gromov hyperbolic space $S$ is proper, i.e., some closed balls of $S$ might not be compact.
\end{remark}

We will use the notation $[s,t]$ to denote a geodesic segment between two points $s,t\in S$. Note that such a geodesic may not be unique. Thus, by $[s,t]$, we mean that we choose one geodesic between $s,t\in S$ and $[s,t]$ will only denote this chosen geodesic. We might specify our choice if necessary, but in most cases we will not do so and just choose an arbitrary geodesic implicitly.

\subsection{Gromov product and Gromov boundary}\label{gpgb}

We recall the notions of Gromov products and Gromov boundaries. Our main references are Bridson-Haefliger \cite{BH}, V\"ais\"al\"a \cite{hyperbolic space}. We shall also prove certain properties of these objects which will be useful later in this article.

Let $S$ be a $\delta$-hyperbolic space. The Gromov product of $x,y$ with respect to $z$, denoted by $(x\cdot y)_z$, where $x,y,z\in S$, is defined by 
$$(x\cdot y)_z=(d(x,z)+d(y,z)-d(x,y))/2.$$

One can reformulate Gromov hyperbolicity by using the Gromov product. In particular, we will use the following inequality many times later in this article. It can be easily extracted from the proofs of Propositions 1.17 and 1.22 in Chapter III.H of \cite{BH}. 

\begin{equation}\label{4.G}
(x\cdot y)_w\geqslant \min \{(x\cdot z)_w, (y\cdot z)_w\}-4\delta,~~~~\forall x,y,z,w\in S.
\end{equation}

Define the Gromov boundary $\partial S$ of $S$ as follows: Pick a point $e\in S$. A sequence of points $\{s_n\}_{n\geqslant 1}\subset S$ is called \textit{converging to $\infty$} if $(s_i\cdot s_j)_e\rightarrow \infty$ as $i$ and $j$ tend to $\infty$. We say that two sequences $\{x_n\}_{n\geqslant 1},\{y_n\}_{n\geqslant 1}$ converging to $\infty$ are \textit{equivalent} and write $\{x_n\}_{n\geqslant 1}\sim \{y_n\}_{n\geqslant 1}$ if $(x_n\cdot y_n)_e\rightarrow\infty$ as $n\rightarrow\infty$. It follows from \eqref{4.G} that $\sim$ is indeed an equivalence relation. The \textit{Gromov boundary} $\partial S$ is then defined as the set of all sequences in $S$ converging to $\infty$ modulo the equivalence relation $\sim$. Elements of $\partial S$ are equivalence classes of sequences in $S$ converging to $\infty$ and we say that a sequence $\{x_n\}_{n\geqslant 1}\in S$ tends to a boundary point $x\in\partial S$ and write $x_n\rightarrow x$ as $n\rightarrow \infty$ if $\{x_n\}_{n\geqslant 1}\in x$.

The definition of the Gromov product can be extended to $S\cup \partial S$. Given $x,y\in S\cup\partial S$, if $x\in S,y\in \partial S$, define $(x\cdot y)_e$ by
$$(x\cdot y)_e=\inf\{\liminf_{n\rightarrow\infty}(x\cdot y_n)_e\},$$
where the infimum is taken over all sequences $\{y_n\}_{n\geqslant 1}$ tending to $y$; if $x\in\partial S,y\in S$, then we define $(x\cdot y)_e$ by flipping the role of $x,y$ in the last equality; finally, if $x,y\in\partial S$, define $(x\cdot y)_e$ by
$$(x\cdot y)_e=\inf \{\liminf_{i,j\rightarrow\infty}(x_i\cdot y_j)_e\},$$
where the infimum is taken over all sequences $\{x_n\}_{n\geqslant 1}$ tending to $x$ and $\{y_n\}_{n\geqslant 1}$ tending to $y$.

Given a positive number $\zeta$. For $s,t\in \partial S$, let
$$d'(s,t)=\exp(-\zeta (s\cdot t)_e),~~~~\rho(s,t)=\inf \sum_{k=1}^{n}d'(s_k,s_{k+1}),$$
where the infimum is taken over all finite sequences $s=s_1,s_2,...,s_{n+1}=t$. By \cite[Proposition 5.16]{hyperbolic space}, if $\zeta$ is small enough, $\rho$ will be a metric for $\partial S$ and $d',\rho$ will satisfy

 \begin{equation} \label{4.0}
 d'(s,t)/2\leqslant \rho(s,t)\leqslant d'(s,t), ~~~~ \forall s,t\in \partial S.
 \end{equation}
 
From now on we will fix a sufficiently small $\zeta$ such that $\rho$ is a metric and that \eqref{4.0} holds.
 
\begin{remark}
We construct $\partial S$ with the help of a chosen point $e$, but the Gromov boundary does not depend on the choice, i.e., we can pick another point $e'\in S$ and use the same procedure to produce a Gromov boundary of $S$ with respect to $e'$. The two resulted boundaries can be naturally identified.

Note that $\rho$ induces a topology $\tau$ on $\partial S$. While $\rho$ does depend on the point $e$ and the constant $\zeta$ we choose, $\tau$ is independent of those choices and thus we get a canonical topology on $\partial S$. In the sequel, the topological concepts of $\partial S$ (for example, open sets) are the ones with respect to this canonical topology.
\end{remark}

For $x\in S$ and $K\in \mathbb{R}$, we employ the notation
$$U_K(x)=\{s\in S\mid (x\cdot s)_e>K\}.$$

Also recall that $B_S(x,r)$ denotes the open ball in $S$ centered at $x$ with radius $r$ and that $[u,v]$ denotes a geodesic segment between $u,v\in S$.

The following estimates \ref{4.1}-\ref{4.4} are well-known properties of hyperbolic spaces and Gromov products. For proofs, the readers are referred to \cite{hyperbolic space}.

\begin{lemma}\label{4.1}
Let $x,y$ be two distinct points of $\partial S$. Then there exist $K>0$ such that for every $u\in U_K(x), v\in U_K(y)$, we have $|(u\cdot v)_e-(x\cdot y)_e|<12\delta$.
\end{lemma}

\begin{lemma}\label{4.2}
Let $u,v$ be two points of $S$. Then $d(e,[u,v])-8\delta\leqslant(u\cdot v)_e\leqslant d(e,[u,v])$.
\end{lemma}

A direct consequence of Lemma \ref{4.2} is:

\begin{lemma}\label{4.3}
Let $u,v$ be two points of $S$ and let $w\in [u,v]$, then $(u\cdot w)_e\geqslant (u\cdot v)_e-8\delta$.
\end{lemma}

Combine Lemmas \ref{4.1}, \ref{4.2}:

\begin{lemma}\label{4.4}
Let $x,y$ be two distinct points of $\partial S$. Then there exist $K>0$ such that $|d(e,[u,v])-(x\cdot y)_e|<20\delta$ for every $u\in U_K(x), v\in U_K(y)$.
\end{lemma}

\begin{lemma}\label{2.5}
Let $x,y$ be two points of $\partial S$ such that $(x\cdot y)_e>K$ for some number $K$. Suppose $\{x_n\}_{n\geqslant 1}$ is a sequence in $S$ tending to $x$. Then there exists $N>0$ such that $(x_n\cdot y)_e>K$ for all $n\geqslant N$.
\end{lemma}

\begin{proof}
Fix $\epsilon>0$ such that $(x\cdot y)_e>K+\epsilon$. Let $\{y_n\}_{n\geqslant 1}$ be any sequence in $S$ tending to $y$. By the definition of $(x\cdot y)_e$,
$$\liminf_{m,n\rightarrow\infty}(x_m\cdot y_n)_e\geqslant (x\cdot y)_e>K+\epsilon.$$
Thus, there exists $N>0$ such that $(x_n\cdot y_m)_e>K+\epsilon$ for all $m,n\geqslant N$. In particular,
$$\liminf_{m\rightarrow\infty}(x_n\cdot y_m)_e\geqslant K+\epsilon$$
for all $n\geqslant N$.

As the above inequality holds for any sequence $\{y_n\}_{n\geqslant 1}$ tending to $y$, we have, for all $n\geqslant N$, $(x_n\cdot y)\geqslant K+\epsilon>K$.
\end{proof}

\begin{lemma}\label{2.6}
Let $x,y$ be two distinct points of $\partial S$. Then there exist $D,K>0$ such that for every $u\in U_K(x),v\in U_K(y)$, we have $d(e,[u,v])<D$ .
\end{lemma}

\begin{proof}
Since $x\neq y$, there exists $D>0$ such that $(x\cdot y)_e<D-20\delta$. By Lemma \ref{4.4}, we can pick $K>0$ large enough so that $d(e,[u,v])<D$ for every $u\in U_K(x),v\in U_K(y)$.
\end{proof}

\begin{lemma}\label{2.7}
Let $x$ be a point of $\partial S$. Then for every $R>0$, there exists $K>0$ such that $d(e,U_K(x))>R$.
\end{lemma}

\begin{proof}
We only need to prove that for every $R>0$, $(x\cdot z)_e<R$ for all $z\in B_S(e,R)$. Fix any $z\in B_S(e,R)$. Let $\{x_n\}_{n\geqslant 1}$ be any sequence in $S$ tending to $x$ as $n\rightarrow \infty$. By Lemma \ref{4.2}, $\liminf_{n\rightarrow\infty}(x_n\cdot z)_e\leqslant \liminf_{n\rightarrow\infty}d(e,[x_n,z])\leqslant d(e,z)<R$. As $\{x_n\}_{n\geqslant 1}$ is arbitrary, we obtain $(x\cdot z)_e<R$. 
\end{proof}

\begin{lemma}\label{2.75}
Let $x$ be a point of $\partial S$. Then for every $R>0$, there exists $K>0$ such that for every $u_1,u_2\in U_K(x)$, we have $[u_1,u_2]\subset U_R(x)$.
\end{lemma}

\begin{proof}
Let $K=R+17\delta$ and let $u_1,u_2$ be two points of $U_K(x)$. We first prove that $(u_1\cdot u_2)_e>R+13\delta$. Let $\{x_n\}_{n\geqslant 1}$ be any sequence in $S$ tending to $x$. By \eqref{4.G},
$$(u_1\cdot u_2)_e\geqslant\min\{(u_1\cdot x_n)_e,(u_2\cdot x_n)_e\}-4\delta$$
for all $n$. Pass to a limit and we obtain $(u_1\cdot u_2)>K-4\delta=R+13\delta$.

Let $t$ be any point of $[u_1,u_2]$. As $(u_1\cdot u_2)_e>R+13\delta$, we have $(u_1\cdot t)_e>R+5\delta$ by Lemma \ref{4.3}. By \eqref{4.G} again,
$$(t\cdot x_n)_e\geqslant \min\{(t\cdot u_1)_e,(u_1\cdot x_n)_e\}-4\delta$$
for all $n$. By passing to a limit and the arbitrariness of $\{x_n\}_{n\geqslant 1}$, we obtain,
$(t\cdot x)_e\geqslant R+\delta>R$ and thus $t\in U_R$. 
\end{proof}

\begin{lemma}\label{2.8}
Let $x,y$ be two distinct points of $\partial S$. Then for every $R>0$, there exists $K>0$ such that for every $u\in U_K(x),v\in U_K(y)$, we have $d(u,v)>R$.
\end{lemma}

\begin{proof}
Given any $R>0$, by Lemmas \ref{2.6} and \ref{2.7}, if $K$ is large enough, we will have $d(e,[u,v])<D$ and that $d(e,u)>R+D$, for every $u\in U_K(x),v\in U_K(y)$. Fix one such $K$ and let $u\in U_K(x),v\in U_K(y)$. Select $t\in [u,v]$ such that $d(e,t)=d(e,[u,v])$ by the compactness of $[u,v]$. Then $d(u,v)\geqslant d(u,t)\geqslant d(u,e)-d(e,t)>R$, as desired. 
\end{proof}

\begin{proposition}\label{2.85}
Let $x,y$ be two distinct points of $\partial S$. Then for every $R>0$, there exists $K>0$ such that for every $u_1,u_2\in U_K(x)$ and every $v_1,v_2\in U_K(y)$, we have $d([u_1,u_2],[v_1,v_2])>R$.
\end{proposition}

\begin{proof}
Given any $R>0$, by Lemma \ref{2.8}, there exists $K'>0$ such that for every $u\in U_{K'}(x),v\in U_{K'}(y)$, we have $d(u,v)>R$. By Lemma \ref{2.75}, there exists $K>0$ such that $[u_1,u_2]\subset U_{K'}(x),[v_1,v_2]\subset U_{K'}(y)$ for every $u_1,u_2\in U_K(x)$ and $v_1,v_2\in U_K(y)$. It follows that $d([u_1,u_2],[v_1,v_2])>R$ for every $u_1,u_2\in U_K(x)$ and $v_1,v_2\in U_K(y)$.
\end{proof}

\begin{lemma}\label{2.9}
Let $x,y$ be two distinct points of $\partial S$. Then there exists $D>0$ with the following property: 

For every $K>D$, there exists $R>0$ such that for every $u\in U_R(x),v\in U_R(y)$ and every $t\in [u,v]\backslash B_S(e,K)$, we have
$$\max\{(t\cdot x)_e,(t\cdot y)_e\}>K-D-12\delta.$$
\end{lemma}

\begin{proof}
Use Lemma \ref{2.6} and pick $D>0$ and $R>K$ such that $d(e,[u,v])<D$ for every $u\in U_R(x),v\in U_R(y)$. Fix $u\in U_R(x),v\in U_R(y)$. Let $t\in [u,v]\backslash B_S(e,K)$, let $[t,u]$ (resp. $[t,v]$) be the subgeodesic of $[u,v]$ from $u$ to $t$ (resp. from $t$ to $v$) and pick $s\in [u,v]$ such that $d(e,s)=d(e,[u,v])$ by the compactness of $[u,v]$. Without loss of generality, we may assume that $s\not\in [t,u]$.

We prove that $[t,u]\cap B_S(e,K-D)=\emptyset$ by contradiction. Suppose there is some $z$ belonging to $[t,u]\cap B_S(e,K-D)$. As $t\not\in B_S(e,K)$, we have $d(t,z)\geqslant d(t,e)-d(e,z)>D$. As $d(e,t)>K$ and $d(e,s)<D$, we have $d(s,t)\geqslant d(e,t)-d(e,s)=K-D$. Thus, 
$$d(s,z)=d(s,t)+d(t,z)>K-D+D=K.$$
But $d(s,z)\leqslant d(s,e)+d(e,z)<D+K-D=K$, a contradiction.

Apply Lemma \ref{4.2} and we see that $(t\cdot u)_e>K-D-8\delta$. Let $\{x_n\}_{n\geqslant 1}$ be a sequence in $S$ tending to $x$. By \eqref{4.G}, $(t\cdot x_n)_e\geqslant\min\{(u\cdot x_n)_e,(t\cdot u)_e\}-4\delta$ for all $n$. Pass to a limit and we obtain $(t\cdot x)_e>\min\{R,K-D-8\delta\}-4\delta=K-D-12\delta$.
\end{proof}

\begin{lemma}\label{2.10}
Let $x,y$ be two points of $\partial S$. Then for $K>0$, $u_1,u_2\in U_{K+6\delta}(x)$ and $v_1,v_2\in U_{K+6\delta}(y)$, $[u_1,v_1]\cap B_S(e,K)$ lies inside the $2\delta$-neighborhood of $[u_2,v_2]$.
\end{lemma}  

\begin{proof}
Given any $K>0$, fix  $u_1,u_2\in U_{K+6\delta}(x)$ and $v_1,v_2\in U_{K+6\delta}(y)$. Let $\{x_n\}_{n\geqslant 1}$ be a sequence in $S$ tending to $x$. By \eqref{4.G}, $(u_1\cdot u_2)_e\geqslant \{(u_1\cdot x_n)_e,(u_2\cdot x_n)_e\}-4\delta$ for all $n$. Pass to a limit and we obtain $(u_1\cdot u_2)>K+2\delta$. By Lemma \ref{4.2}, $d(e,[u_1,u_2])> K+2\delta$. Similarly, $d(e,[v_1,v_2])> K+2\delta$.

Consider the geodesic quadrilateral: $[u_1,u_2],[u_2,v_2],[v_2,v_1],[v_1,u_1]$. By hyperbolicity, $B_S(e,K)\cap [u_1,v_1]$ lies inside the $2\delta$-neighborhood of $[u_1,u_2]\cup[u_2,v_2]\cup[v_2,v_1]$. Since $d(e,[u_1,u_2])>K+2\delta$, we have $d([u_1,u_2],B_S(e,K))>2\delta$ by the triangle inequality. Likewise, $d([v_1,v_2],B_S(e,K))>2\delta$. It follows that $[u_1,v_1]\cap B_S(e,K)$ lies inside the $2\delta$-neighborhood of $[u_2,v_2]$.
\end{proof}

\begin{lemma}\label{2.11}
Let $x,y,z$ be three distinct points of $\partial S$. Then for every $K>0$, there exists $R>0$ such that for every $u\in U_R(x),v\in U_R(y),w\in U_R(z)$, we have $d(w,[u,v])>K$.
\end{lemma}

\begin{proof}
By Lemmas \ref{2.6} and \ref{2.8}, there exists $D>0$ with the following property: Given any $K>0$, there exists $R'>0$ such that 
$$\max\{d(e,[u,w]),d(e,[v,w])\}<D, ~~~~\min\{d(u,w),d(v,w)\}>K$$
for all $u\in U_{R'}(x),v\in U_{R'}(y),w\in U_{R'}(z)$. By Lemmas \ref{2.7} and \ref{2.9}, there exists $R>R'$ such that 
$$[u,v]\backslash B_S(e,R'+D+12\delta)\subset U_{R'}(x)\cup U_{R'}(y), ~~~~d(w,B_S(e,R'+D+12\delta))>K$$
for all $u\in U_R(x),v\in U_R(y),w\in U_R(z)$.

Fix arbitrary $u\in U_R(x),v\in U_R(y),w\in U_R(z)$. We verify that $d(w,[u,v])>K$. Pick $t\in [u,v]$ such that $d(w,t)=d(w,[u,v])$ by the compactness of $[u,v]$. By our choice of $R$, either $t\in U_{R'}(x)\cup U_{R'}(y)$ or $t\in B_S(e,R'+D+12\delta)$. If $t\in U_{R'}(x)$ or $U_{R'}(y)$, then $d(w,[u,v])=d(w,t)>K$ by our choice of $R'$. If $t\in B_S(e,R'+D+12\delta)$, we will still have $d(w,[u,v])=d(w,t)>K$ by our choice of $R$.
\end{proof}

\begin{proposition}\label{2.115}
Let $x,y,z$ be three distinct points of $\partial S$. Then for every $K>0$, there exists $R>0$ such that for every $u\in U_R(x),v\in U_R(y)$ and $w_1,w_2\in U_R(z)$, we have $d([u,v],[w_1,w_2])>K$.
\end{proposition}

\begin{proof}
Given any $K>0$, by Lemma \ref{2.11}, there exists $R'>0$ such that $d(w,[u,v])>K$ for every $u\in U_{R'}(x),v\in U_{R'}(y),w\in U_{R'}(z)$. By Lemma \ref{2.75}, there exists $R>0$ such that $[w_1,w_2]\subset U_{R'}(z)$ for every $w_1,w_2\in U_R(z)$. It follows that $d([u,v],[w_1,w_2])>K$ for every $u\in U_R(x),v\in U_R(y)$ and $w_1,w_2\in U_R(z)$.
\end{proof}

\begin{lemma}\label{2.12}
Let $u,v$ be two points of $S$. Select a geodesic $[u,v]$ connecting $u,v$ and let $T=\{z\in [u,v]\mid d(e,z)\leqslant d(e,[u,v])+42\delta\}$. Then the diameter of $T$ is at most $88\delta$.
\end{lemma}

\begin{proof}
Suppose, for the contrary, that there exists $x,y\in T$ such that $d(x,y)>88\delta$. Let $[x,y]$ be the subgeodesic of $[u,v]$ between $x$ and $y$. Let $t$ be the midpoint of $[x,y]$. Obviously, both $d(x,t)$and $d(y,t)$ are strictly greater than $44\delta$.

Consider the geodesic triangle $[x,e],[e,y],[x,y]$. There is a point $w\in [x,e]\cup[e,y]$ such that $d(t,w)<\delta$. If $w\in [x,e]$, then since $d(x,t)>44\delta$, $d(x,w)>44\delta-\delta>43\delta$ by the triangle inequality, hence
$$d(t,e)\leqslant d(t,w)+d(w,e)<\delta+d(e,[u,v])+42\delta-43\delta<d(e,[u,v]).$$

Similarly, if $w\in [e,y]$, then the same argument with $y$ in place of $x$ shows that $d(y,e)<d(e,[u,v])$. Either case contradicts the definition of $d(e,[u,v])$. 
\end{proof}

\begin{proposition}\label{2.13}
Let $\{p_n\}_{n\geqslant 1},\{q_n\}_{n\geqslant 1},\{r_n\}_{n\geqslant 1},\{s_n\}_{n\geqslant 1}$ be sequences in $S$ tending to four distinct boundary points $p,q,r,s$ respectively. For each $n$, choose a point $a_n$ (resp. $b_n$) in $[p_n,q_n]$ (resp. $[r_n,s_n]$) such that $d(e,a_n)=d(e,[p_n,q_n])$ (resp. $d(e,b_n)=d(e,[r_n,s_n])$) by the compactness of $[p_n,q_n]$ (resp. $[r_n,s_n]$).

If $m, n$ are large enough, $[a_m,b_m]$ will be in the $92\delta$-neighborhood of $[a_n,b_n]$.
\end{proposition}

\begin{figure}
\subcaptionbox{The estimate of $T$}
{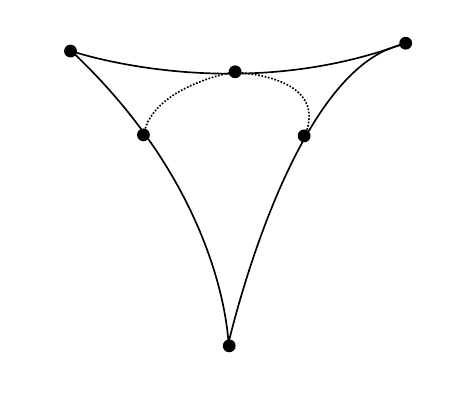}
\subcaptionbox{As $n$ increases, $[a_n,b_n]$ is stablized}
{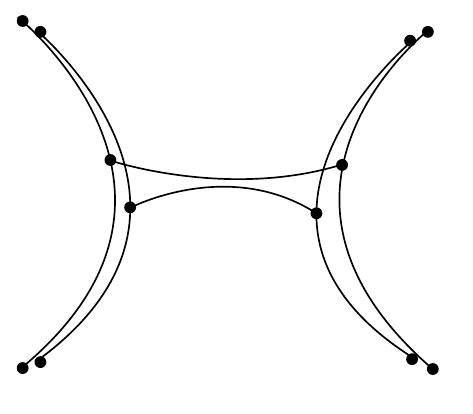}
\caption{Ideas behind Lemma \ref{2.12} and Proposition \ref{2.13}}
\end{figure}

\begin{proof}
By Lemma \ref{4.4}, there exists $N_1$ such that if $n>N_1$, $$|d(e,[p_n,q_n])-(p\cdot q)_e|<20\delta.$$
There exists $N_2$ such that if $m, n>N_2$, both $d(e, [p_m,p_n])$ and $d(e,[q_m,q_n])$ will be strictly greater than $(p\cdot q)_e+22\delta$, by the fact that $\{p_n\}_{n\geqslant 1},\{q_n\}_{n\geqslant 1}$ are sequences tending to $\infty$ and Lemma \ref{4.2}. Let $m,n>\max\{N_1,N_2\}$ and consider the geodesic quadrilateral consisting of the 4 sides: $[p_m,q_m],[q_m,q_n],[q_n,p_n],[p_n,p_m]$. There is a point $a_{m,n}\in [p_m,q_m]\cup[q_m,q_n]\cup[p_n,p_m]$ such that $d(a_{m,n},a_n)<2\delta$. Since
$$d(e,a_{m,n})\leqslant d(e,a_n)+2\delta\leqslant (p\cdot q)_e+22\delta,$$
we have $a_{m,n}\in [p_m,q_m].$ We already know that both $|d(e,[p_n,q_n])-(p\cdot q)_e|$ and $|d(e,[p_m,q_m])-(p\cdot q)_e|$ are less than or equal to $20\delta$. Therefore, $$|d(a_n,e)-d(e,[p_m,q_m])|=|d(e,[p_n,q_n])-d(e,[p_m,q_m])|\leqslant 40\delta.$$
The triangle inequality implies that 
$$d(a_{m,n},e)-d(e,[p_m,q_m])\leqslant d(a_{m,n},a_n)+d(a_n,e)-d(e,[p_m,q_m])\leqslant 42\delta.$$
By Lemma \ref{2.12}, $d(a_{m,n},a_m)\leqslant 88\delta$, thus $d(a_m,a_n)\leqslant (88+2)\delta=90\delta$. Similarly, there exists $N_3>0$ such that if $m,n>N_3$, $d(b_m,b_n)\leqslant 90\delta$. Now let $$m,n>\max\{N_1,N_2,N_3\}$$
and consider the geodesic quadrilateral $[a_m,a_n],[a_n,b_n],[b_n,b_m],[b_m,a_m]$. Every point of $[a_m,b_m]$ is $2\delta$-close to a point in the union of the other three sides, which is $90\delta$-close to $[a_n,b_n]$, thus $[a_m,b_m]$ is in the $92\delta$-neighborhood of $[a_n,b_n]$.
\end{proof}

\section{Group actions on Gromov hyperbolic spaces}\label{sec.gahs}
\subsection{Acylindrically hyperbolic groups}\label{subsec.ah}

Let $(S,d)$ be a Gromov hyperbolic space and let $G$ be a group acting on $S$ by isometries. The action of $G$ is called \textit{acylindrical} if for every $\epsilon>0$ there exist $R, N>0$ such that for every two points $x,\ y$ with $d(x,y)\geqslant R$, there are at most $N$ elements $g\in G$ satisfying both $d(x,gx)\leqslant\epsilon$ and $d(y,gy)\leqslant\epsilon$. The \textit{limit set} $\Lambda(G)$ of $G$ on $\partial S$ is the set of limit points in $\partial S$ of a $G$-orbit in $S$, i.e., 
$$\Lambda(G)=\{x\in\partial S\mid \text{ there exists a sequence in } Gs \text{ tending to } x, \text{ for some }s\in S \}.$$
If $\Lambda(G)$ contains at least three points, we say the action of $G$ is \textit{non-elementary}. Acylindrically hyperbolic groups are defined by Osin \cite{acylindrically hyperbolic group}: 

\begin{definition}
A group $G$ is called \textit{acylindrically hyperbolic} if it admits a non-elementary acylindrical action by isometries on a Gromov hyperbolic space.
\end{definition}

\begin{theorem}\label{acy}
\upshape (Osin \cite[Theorem 1.2]{acylindrically hyperbolic group}) \itshape For a group $G$, the following are equivalent.

\begin{enumerate}[leftmargin=3em]
\item[(AH$_1$)] $G$ admits a non-elementary acylindrical and isometric action on a Gromov hyperbolic space.

\item[(AH$_2$)] $G$ is not virtually cyclic and admits an isometric action on a Gromov hyperbolic space such that at least one element of $G$ is loxodromic and satisfies the WPD condition.

\end{enumerate}
\end{theorem}

Recall that an element $g\in G$ is called \textit{loxodromic} if the map $\mathbb{Z}\rightarrow S,\ n\mapsto g^{n}s$ is a quasi-isometric embedding for some (equivalently, any) $s\in S$. The WPD condition, due to Bestvina-Fujiwara \cite{BF}, is defined as follows:

\begin{definition}
Let $G$ be a group acting isometrically on a Gromov hyperbolic space $(S,d)$ and let $g$ be an element of $G$. One says that $g$ satisfies the \textit{weak proper discontinuity} condition (or $g$ is a \textit{WPD} element) if for every $\epsilon>0$ and every $s\in S$, there exists $K\in\mathbb{N}$ such that 
\begin{center}
$|\{h\in G\mid d(s,hs)<\epsilon,d(g^{K}s,hg^{K}s)<\epsilon\}|<\infty.$
\end{center}
\end{definition}

In fact, $g$ satisfies the WPD condition for every $s$ if and only if $g$ satisfies the same condition for just one $s\in S$. More precisely, let us consider the following condition.

\begin{enumerate}[leftmargin=2em]
\item[$(\bigstar)$] There is a point $s\in S$ such that for every $\epsilon>0$, there exists $K\in\mathbb{N}$ with
\begin{center}
$|\{h\in G\mid d(s,gs)<\epsilon,d(g^{K}s,hg^{K}s)<\epsilon\}|<\infty.$
\end{center}
\end{enumerate}

\begin{lemma}\label{star}
Let $G$ be a group acting isometrically on a Gromov hyperbolic space $(S,d)$ and let $g$ be an element of $G$, then $g$ satisfies the WPD condition if and only if $g$ satisfies $(\bigstar)$.
\end{lemma}

\begin{proof}
Clearly, WPD implies $(\bigstar)$. On the other hand, suppose that $g$ satisfies $(\bigstar)$ for some point $s_{0}\in S$, but $g$ does not satisfy the WPD condition. Thus, there is some $\epsilon_{1}>0$ and $s_{1}\in S$ such that for every $K\in\mathbb{N}$, we have
$$|\{h\in G\mid d(s_{1},hs_{1})<\epsilon_{1},d(g^{K}s_{1},hg^{K}s_{1})<\epsilon_{1}\}|=\infty.$$

Let $\epsilon=2d(s_{0},s_{1})+\epsilon_{1}$ and let $K_{0}$ be an integer such that 

\begin{equation}\label{wpdstar}
|\{h\in G\mid d(s_{0},hs_{0})<\epsilon,d(g^{K_{0}}s_{0},hg^{K_{0}}s_{0})<\epsilon\}|<\infty.
\end{equation}

For any element $h\in G$, if $d(s_{1},hs_{1})<\epsilon_{1}$, then 
$$d(s_{0},hs_{0})\leqslant d(s_{0},s_{1})+d(s_{1},hs_{1})+d(hs_{1},hs_{0})<\epsilon.$$

Similarly, if h is an element in $G$ such that $d(g^{K_{0}}s_{1},hg^{K_{0}}s_{1})<\epsilon_{1}$, then 
$$d(g^{K_{0}}s_{0},hg^{K_{0}}s_{0})<\epsilon.$$

As $|\{h\in G\mid d(s_{1},hs_{1})<\epsilon_{1},d(g^{K_{0}}s_{1},hg^{K_{0}}s_{1})<\epsilon_{1}\}|=\infty$, it follows that 
$$|\{h\in G\mid d(s_{0},hs_{0})<\epsilon,d(g^{K_{0}}s_{0},hg^{K_{0}}s_{0})<\epsilon\}|=\infty.$$

This contradicts inequality \eqref{wpdstar}.
\end{proof}

\subsection{Induced actions on Gromov boundaries}\label{sub.iab}
Let $G$ be a group acting isometrically on a Gromov hyperbolic space $(S,d)$. As mentioned in Section \ref{gpgb}, the Gromov boundary $\partial S$ of $S$ is defined via sequences of points in $S$ tending to $\infty$ and there is a canonical topology for $\partial S$. Note that $G$ maps one sequence tending to $\infty$ to another such sequence so it  naturally acts on $\partial S$ and this action is by homeomorphisms (see V\"ais\"al\"a \cite{hyperbolic space} for details). 

If an element $g\in G$ is loxodromic, then $\{g^{-n}e\}_{n\geqslant 1}$ and $\{g^{n}e\}_{n\geqslant 1}$ are two sequences in $S$ tending to different boundary points $x,y\in \partial S$ respectively, and $g$ fixes these boundary points. Moreover, $g$ actually has the so-called north-south dynamics on $\partial S$. This is well-known when the space $S$ is proper. Nevertheless, the original idea of Gromov \cite{GM} works even for non-proper spaces. The readers are referred to Hamann \cite{no} for a detailed proof.

\begin{definition}\label{3.4}
Let $G$ be a group acting by homeomorphisms on a topological space $M$. We say an element $g\in G$ has \textit{north-south dynamics} on $M$ if the following two conditions are satisfied:

\begin{enumerate}[leftmargin=1.5em]
\item[1.] $g$ fixes exactly two distinct points $x,y\in M$.

\item[2.] For every pair of open sets $U,V$ containing $x,y$ respectively, there exists $N>0$ such that $g^n(M\backslash U)\subset V$ for all $n>N$. 
\end{enumerate}
\end{definition}

\begin{lemma}\label{no.3.4}\upshape (Hamann \cite[Proposition 3.4]{no})
\itshape Suppose that a group $G$ acts isometrically on a Gromov hyperbolic space $S$ and has a loxodromic element $g$. Let $\partial S$ be the Gromov boundary of $S$ with the topology defined in Section \ref{gpgb}. Then, with respect to the action of $G$ on $\partial S$ (induced by the action of $G$ on $S$), $g$ has north-south dynamics.
\end{lemma}

\section{Convergence groups}\label{cg}

Let $G$ be a group acting on a compact metrizable topological space $M$ by homeomorphisms (with respect to the topology induced by the metric $d$). We assume that both $G$ and $M$ are infinite sets since otherwise the notion of convergence groups will be trivial. $G$ is called a \textit{discrete convergence group} if for every infinite sequence $\{g_{n}\}_{n\geqslant 1}$ of distinct elements of $G$, there exists a subsequence $\{g_{n_{k}}\}$ and points $a,b\in M$ such that $g_{n_{k}}|_{M\backslash \{a\}}$ converges to $b$ locally uniformly, that is, for every compact set $K\subset M\backslash\{a\}$ and every open neighborhood $U$ of $b$, there is an $N\in\mathbb{N}$ such that $g_{n_{k}}(K)\subset U$ whenever $n_{k}>N$. In what follows, when we say a group $G$ is a convergence group, we always mean that $G$ is a discrete convergence group, and we will call the action of $G$ on $M$ a convergence action.

An equivalent definition of a convergence action can be formulated in terms of the action on the space of distinct triples. Let 
$$\Theta_{3}(M)=\{(x_{1},x_{2},x_{3})\in M\mid x_{1}\neq x_{2},\ x_{2}\neq x_{3},\ x_{1}\neq x_{3}\}$$ 
be the set of distinct triples of points in $M$, endowed with the subspace topology induced by the product topology of $M^{3}$. Notice that $\Theta_{3}(M)$ is non-compact with respect to this topology. Clearly, the action of $G$ on $M$ naturally induces an action of $G$ on $\Theta_{3}(M): (x_{1},x_{2},x_{3})\rightarrow (gx_{1},gx_{2},gx_{3})$, for all $g\in G$.

\begin{proposition}\label{cg, 1.1}
\upshape (Bowditch \cite[Proposition 1.1]{convergence group}) \itshape The action of $G$ on $M$ is a convergence action if and only if the action of $G$ on $\Theta_{3}(M)$ is properly discontinuous, that is, for every compact set $K\subset \Theta_{3}(M)$, there are only finitely many elements $g\in G$ such that $gK\cap K\neq \emptyset$. 
\end{proposition}

\begin{remark}\label{pre of cg}
Let $G$ be a convergence group acting on a compact metrizable topological space $M$. Elements of $G$ can be classified into the following three types:

\begin{itemize}[leftmargin=1.5em]
\item\textit{Elliptic}: having finite order;

\item\textit{Parabolic}: having infinite order and fixing a unique point of $M$;

\item\textit{Loxodromic}: having infinite order and fixing exactly two points of $M$.
\end{itemize}

Moreover, a parabolic element cannot share its fixed point with a loxodromic element (see Tukia \cite[Theorem 2G]{property}).
\end{remark}



A convergence group $G$ is called \textit{elementary} if it preserves setwise a nonempty subset of $M$ with at most two elements. The next theorem is a combination of several results in \cite{property} (Theorems 2S, 2U and 2T):

\begin{theorem}\label{pre}
\upshape (Tukia, 1994) \itshape If $G$ is a non-elementary convergence group acting on a compact metrizable topological space $M$, the following statements hold.

\begin{enumerate}

\item[(1)] $G$ contains a non-abelian free group as its subgroup and thus cannot be virtually abelian.

\item[(2)] There is an element $g\in G$ having north-south dynamics on $M$.

\end{enumerate}
\end{theorem}

For more information on convergence groups, the readers are referred to Bowditch \cite{convergence group}, Tukia \cite{property}.

\section{The (C) Condition}\label{sec.gcg}
In this section, we prove some properties of Condition (C) (see Definition \ref{3.5}).

%

\begin{lemma}\label{c-gc}
Let $G$ be a convergence group acting on a compact metrizable topological space $M$. Then this action satisfies (C).
\end{lemma}

\begin{proof}
Let $u=(x,x),v=(y,y)$ be two distinct points on the diagonal $\Delta$ of $M^2$ (hence $x,y\in M$ and $x\neq y$), let $d$ be a metric on $M$ compatible with its topology, and let
$$U=B_M(x,d(x,y)/3)\times B_M(x,d(x,y)/3),~~~~V=B_M(y,d(x,y)/3)\times B_M(y,d(x,y)/3).$$
Then $U,V$ are open sets in $M^2$ containing $u,v$, respectively. Let us check (C) for $U,V$.

Let $a=(p,q),b=(r,s)$ be two distinct points of $M^2\backslash\Delta$ (hence $p,q,r,s\in M$,$p\neq q,r\neq s$) and let
$$A=B_M(p,d(p,q)/3)\times B_M(q,d(p,q)/3),~~~~B=B_M(r,d(r,s)/3)\times B_M(s,d(r,s)/3).$$
Then $A,B$ are open sets in $M^2$ containing $a,b$, respectively. Suppose that (C) does not hold for $U,V$. Then there exists an infinite sequence $\{g_n\}_{n=1}^{\infty}$ of distinct elements of $G$ such that $g_nA\cap U\neq \emptyset,g_nB\cap V\neq\emptyset$ for all $n\geqslant 1$. In other words,
$$g^{-1}_nU\cap A\neq\emptyset,~~~~g^{-1}_nV\cap B\neq\emptyset$$
for all $n\geqslant 1$.

Consider the infinite sequence $\{g^{-1}_n\}_{n=1}^{\infty}$ of distinct elements of $G$. By the convergence property and passing to a subsequence, one may assume that there exists two points $z,t\in M$ such that $g^{-1}_n|_{M\backslash\{z\}}$ converges to $t$ locally uniformly. By the triangle inequality, we have $d(z,x)+d(z,y)\geqslant d(x,y)$, and thus at least one of $d(z,x)$ and $d(z,y)$ is strictly greater than $d(x,y)/3$. Without loss of generality, we may assume that $d(z,x)>d(x,y)/3$. As $g^{-1}_n|_{M\backslash\{z\}}$ converges to $t$ locally uniformly, there exists a positive integer $N$ such that
$$g^{-1}_N(B_M(x,d(x,y)/3))\subset B_M(t,d(p,q)/6).$$

Note that $g^{-1}_NU\cap A\neq\emptyset$. As a consequence, one has
$$g^{-1}_N(B_M(x,d(x,y)/3))\cap B_M(p,d(p,q)/3)\neq\emptyset,$$
$$g^{-1}_N(B_M(x,d(x,y)/3))\cap B_M(q,d(p,q)/3)\neq\emptyset,$$
and thus
\begin{equation}\label{t1}
B_M(t,d(p,q)/6)\cap B_M(p,d(p,q)/3)\neq\emptyset,
\end{equation}
\begin{equation}\label{t2}
B_M(t,d(p,q)/6)\cap B_M(q,d(p,q)/3)\neq\emptyset,
\end{equation}

\eqref{t1} (resp. \eqref{t2}) imples that $d(t,p)<d(p,q)/6+d(p,q)/3=d(p,q)/2$ (resp. $d(t,q)<d(p,q)/6+d(p,q)/3=d(p,q)/2$). Thus $d(p,q)\leqslant d(t,p)+d(t,q)<d(p,q)/2+d(p,q)/2=d(p,q)$, a contradiction.
\end{proof}

\begin{remark}\label{reformulate}
Let $G$ be a group acting on a topological space $M$ satisfying Condition (C). In order to prepare for the proof of Theorem \ref{1.1}, let us reformulate Definition \ref{3.5} in terms of the action of $G$ on $M$ instead of $M^2$. Let $u=(x,x),v=(y,y)$ be two distinct points on the diagonal $\Delta$ of $M^2$ (hence $x,y\in M$ and $x\neq y$). Condition (C) requires the existence of open sets $U,V$ in $M^2$ containing $u,v$ respectively, with certain properties. By shrinking $U,V$ if necessary, let us assume that $U=X\times X,V=Y\times Y$, where $X,Y$ are open sets in $M$ containing $x,y$ respectively. Suppose that $a,b$ are two distinct points of $M^2\backslash\Delta$. There are several cases to consider.

\noindent\textbf{Case 1:} The coordinates of $a,b$ involve only two distinct points of $M$, i.e., $a=(p,q)$ and $b=(q,p)$, where $p,q$ are distinct points of $M$ (note that $a$ and $b$ are different points of $M^2$ as $M^2$ is the set of ordered pairs of $M$).

Condition (C) asserts the existence of open sets $A,B$ in $M^2$ containing $a,b$ respectively, with certain properties. By shrinking $A,B$ if necessary, we may assume that $A=A_1\times A_2,B=A_2\times A_1$, where $A_1,A_2$ are open sets in $M$ containing $p,q$, respectively. Then (C) can be rephrased as
$$|\{g\in G\mid gA_1\cap X,gA_1\cap Y,gA_2\cap X,gA_2\cap Y\text{ are all non-empty}\}|<\infty.$$

\noindent\textbf{Case 2:} The coordinates of $a,b$ involve only three distinct points of $M$. For example, $a=(p,q)$ and $b=(p,r)$, where $p,q,r$ are three distinct points of $M$.

Again, Condition (C) asserts the existence of certain open sets $A,B$, and one can assume that $A=A_1\times A_2,B=A_1\times B_2$, where $A_1,A_2,B_2$ are open sets in $M$ containing $p,q,r$, respectively. In this case, (C) can be rephrased as
$$|\{g\in G\mid gA_1\cap X,gA_1\cap Y,gA_2\cap X,gB_2\cap Y\text{ are all non-empty}\}|<\infty.$$

The other cases where the coordinates of $a,b$ involve only three distinct points of $M$ can be treated in the same way.

\noindent\textbf{Case 3:} The coordinates of $a,b$ involve four distinct points of $M$, i.e.,$a=(p,q)$ and $b=(r,s)$, where $p,q,r,s$ are four distinct points of $M$.

Once again, Condition (C) asserts the existence of certain open sets $A,B$, and one can assume that $A=A_1\times A_2,B=B_1\times B_2$, where $A_1,A_2,B_1,B_2$ are open sets in $M$ containing $p,q,r,s$, respectively. In this case, (C) can be rephrased as
$$|\{g\in G\mid gA_1\cap X,gA_2\cap X,gB_1\cap Y,gB_2\cap Y\text{ are all non-empty}\}|<\infty.$$
\end{remark}

For further reference, let us sum up the above discussion.

\begin{lemma}\label{reformulate1}
Let $G$ be a group acting on a topological space $M$ which has at least $3$ points. Then this action of $G$ satisfies (C) if and only if for every pair of distinct points $x,y\in M$, there exists open sets $U,V$, in the topological space $M$, containing $x,y$ respectively and satisfying the following (C$_1$), (C$_2$), and (C$_3$) simultaneously.
\begin{enumerate}[leftmargin=2.2em]
\item[(C$_1$)] For every pair of distinct points $p,q\in M$, there exist open sets $A,B$, of the topological space $M$, containing $p,q$ respectively, with
$$|\{g\in G\mid gA\cap U,gA\cap V,gB\cap U,gB\cap V\text{ are all non-empty}\}|<\infty.$$
\item[(C$_2$)] For every three distinct points $p,q,r\in M$, there exist open sets $A,B,C$, of the topological space $M$, containing $p,q,r$ respectively, with
$$|\{g\in G\mid gA\cap U,gB\cap V,gC\cap U,gC\cap V\text{ are all non-empty}\}|<\infty.$$
\item[(C$_3$)] For every four distinct points $p,q,r,s\in M$, there exist open sets $A,B,C,D$, of the topological space $M$, containing $p,q,r,s$ respectively, with
$$|\{g\in G\mid gA\cap U,gB\cap U,gC\cap V,gD\cap V\text{ are all non-empty}\}|<\infty.$$
\end{enumerate}
\end{lemma}

In the rest of this paper, we say that a pair of distinct points $x,y\in M$ \textit{satisfy (C$_1$)} (resp. \textit{(C$_2$), (C$_3$)}) if there exist open sets $U,V$, in the topological space $M$, containing $x,y$ respectively and satisfying (C$_1$) (resp. (C$_2$), (C$_3$)).

\section{Annulus system and hyperbolicity}\label{sec.a}

Throughout this section, let $G$ be a group acting on a topological space $M$. In Section \ref{sec.p}, we are going to prove that if the action $G\curvearrowright M$ satisfies Condition (C) and there exists $g\in G$ such that $g$ has north-south dynamics on $M$, then $G$ admits an isometric action on some Gromov hyperbolic space with $g$ being a loxodromic WPD element (which implies that $G$ is either acylindrically hyperbolic or virtually cyclic, by Theorem \ref{acy}). The proof relies on a construction of Bowditch \cite{construction} called an annulus system, which is surveyed below.

\begin{definition}
An \textit{annulus}, $A$, is an ordered pair, $(A^{-},A^{+})$, of disjoint closed subsets of $M$ such that $M\backslash (A^{-}\cup A^{+})\neq \emptyset$. 
\end{definition}

For an annulus $A$ and $g\in G$, we write $gA$ for the annulus $(gA^{-},gA^{+})$.

An \textit{annulus system} on $M$ is a set of annuli. The system is called \textit{symmetric} if $-A:=(A^{+},A^{-})\in\mathcal{A}$ whenever $A\in\mathcal{A}$.

Let $A$ be an annulus. Given any subset $K\subset M$, we write $K<A$ if $K\subset int A^{-}$ and write $A<K$ if $K\subset int A^{+}$, where $int A^-$ (resp. $int A^+$) denotes the interior of $A^-$ (resp. $A^+$). Thus $A<K$ if and only if $K<-A$. If $B$ is another annulus, we write $A<B$ if $int A^{+}\cup int B^{-}=M$.

Given an annulus system $\mathcal{A}$ on $M$ and $K,L\subset M$, define $(K|L)=n\in\{0,1,...,\infty\}$, where $n$ is the supremum of all positive integers $m$ such that there exist $m$ annuli $A_1,...,A_m$ in $\mathcal{A}$ with $K<A_{1}<A_{2}<...<A_m<L$ (if no such $m$ exists, set $(K|L)=0$). For finite sets we drop braces and write $(a,b|c,d)$ to mean $(\{a,b\}|\{c,d\})$. This gives us a well-defined function $M^{4}\rightarrow [0,+\infty]$. Note that this function is $G$-invariant, i.e., $(gx,gy|gz,gw)=(x,y|z,w)$, for all $g\in G$, provided that the annulus system $\mathcal{A}$ is $G$-invariant.

\begin{definition}
The function from $M^{4}$ to $[0,+\infty]$, defined as above, is called the \textit{crossratio} associated with $\mathcal{A}$.
\end{definition}

Recall the definition of a quasimetric on a set $Q$:

\begin{definition}
Given $r\geqslant 0$, an $r$\textit{-quasimetric} $\rho$ on  a set $Q$ is a function $\rho: Q^{2}\rightarrow [0,+\infty)$ satisfying $\rho(x,x)=0,\ \rho(x,y)=\rho(y,x)$ and $\rho(x,y)\leqslant \rho(x,z)+\rho(z,y)+r$ for all $x,y,z\in Q$.
\end{definition}

A \textit{quasimetric} is an $r$-quasimetric for some $r\geqslant0$. Given $s\geqslant 0$ and a quasimetric space $(Q,\rho)$, an $s$\textit{-geodesic segement} is a finite sequence of points $x_{0},x_{1},...,x_{n}$ such that $-s\leqslant\rho(x_{i},x_{j})-|i-j|\leqslant s$ for all $0\leqslant i,j\leqslant n$. A quasimetric is a \textit{path quasimetric} if there exists $s\geqslant 0$ such that every pair of points can be connected by an $s$-geodesic segment. A quasimetric is called a \textit{hyperbolic quasi-metric} if there is some $k\geqslant 0$ such the 4-point definition of $k$-hyperbolicity holds via the Gromov product (see Bridson-Haefliger \cite[Chapter III. H, Definition 1.20]{BH}). 

Given an annulus system $\mathcal{A}$ on $M$, one can construct a quasimetric on $\Theta_{3}(M)$ from the crossratio associated with $\mathcal{A}$, where
$$\Theta_3(M)=\{(x_1,x_2,x_3)\in M^3\mid x_1\neq x_2,x_2\neq x_3,x_3\neq x_1\}$$
is the set of distinct triples of $M$. Let $x=(x_{1},x_{2},x_{3})$ and $y=(y_{1},y_{2},y_{3})$ be two points of $\Theta_{3}(M)$. Define the function $\rho: (\Theta_{3}(M))^{2}\rightarrow [0,+\infty]$ by 
$$\rho(x,y)=\max (x_{i},x_{j}|y_{k},y_{l}),$$ 
where $(.,.|.,.)$ denotes the crossratio associated with $\mathcal{A}$ and the maximum is taken over all $i,j,k,l\in\{1,2,3\}$ with $i\neq j$ and $k\neq l$. 

Consider two axioms on the crossratio $(.,.|.,.)$ (and hence on the annulus system $\mathcal{A}$):

\begin{enumerate}[leftmargin=2.5em]
\item[(A1)] If $x\neq y$ and $z\neq w$, then $(x,y|z,w)<\infty$.
\item[(A2)] There is some $k\geqslant 0$ such that there are no four points $x,y,z,w\in M$ with $(x,y|z,w)>k$ and $(x,z|y,w)>k$.
\end{enumerate}

\begin{proposition}\label{hq}
Suppose that $G$ is a group acting on a topological space $M$, and that $\mathcal{A}$ is a symmetric, $G$-invariant annulus system on $M$ satisfying (A1) and (A2). Then the map $\rho$ defined as above is a hyperbolic $G$-invariant path quasimetric on $\Theta_3(M)$. 
\end{proposition}

By $\rho$ being $G$-invariant, we mean $\rho(gx,gy)=\rho(x,y)$ for all $x,y\in\Theta_{3}(M)$ and $g\in G$.

\begin{proof}
The fact that $\rho$ is a hyperbolic path quasimetric follows from Propositions 4.2, 6.5 and Lemma 4.3 of Bowditch \cite{construction}. Note that Bowditch assumes that $M$ is compact, but he does not use this assumption in the proofs of Propositions 4.2, 6.5 and Lemma 4.3 of \cite{construction}. The fact that $\rho$ is $G$-invariant follows from the fact that $\mathcal{A}$ is $G$-invariant and the relationship between $\rho$ and $\mathcal{A}$.
\end{proof}

Note that Proposition \ref{hq} only produces a space $\Theta_3(M)$ with a $G$-invariant hyperbolic quasimetric $\rho$, but, as mentioned in the beginning of this section, we need to construct an isometric action of $G$ on some Gromov hyperbolic space, which is a geodesic metric space. This can be easily achieved by passing to a geodesic metric space quasi-isometric with $\Theta_3(M)$.


\begin{definition}
Let $(Q,d)$ and $(Q',d')$ be two quasimetric spaces. A map $f:Q\rightarrow Q'$ is called a \textit{quasi-isometry} from $Q$ to $Q'$ if there exist $\lambda, C,D>0$ such that
\begin{enumerate}
\item[(1)] the inequality $d(x,y)/\lambda-C<d'(f(x),f(y))<\lambda d(x,y)+C$ holds for all $x,y\in Q$;
\item[(2)] every point of $Q'$ is within distance $D$ from the image of $f$.
\end{enumerate}
\end{definition}

\begin{proposition}\label{2}
Let $G$ be a group acting on a topological space $M$ and let $\rho$ be a $G$-invariant hyperbolic path quasimetric on $\Theta_{3}(M)$. Then there is a Gromov hyperbolic space $(S,\rho')$ such that $G$ acts isometrically on $S$ and that there is a $G$-equivariant quasi-isometry $f:\Theta_{3}(M)\rightarrow S$.
\end{proposition}

\begin{proof}
The proof can be easily extracted from \cite{construction}. We provide it for convenience of the readers. Let $s$ be a number such that every pair of points in $\Theta_{3}(M)$ can be connected by an $s$-geodesic. Construct the undirected graph $S$ whose vertex set is just $\Theta_{3}(M)$ and two vertices $x,y$ are connected by an edge if $\rho(x,y)\leqslant s+1$. Define a path-metric, $\rho'$, on $S$ by deeming every edge to have unit length. We see that $S$ is connected and that the inclusion $f:\Theta_{3}(M)\hookrightarrow S$ is a quasi-isometry. Since $\rho'(x,y)$ is an integer for every pair of vertices $x,y\in \Theta_3(M)$, $S$ is a geodesic metric space. $\rho'$ is a hyperbolic metric since $\rho$ is hyperbolic and $f$ is a quasi-isometry. Hence, $S$ is a Gromov hyperbolic space. Moreover, the action of $G$ on $\Theta_{3}(M)$ induces an action of $G$ on $S$: for every $g\in G$, $g$ maps a vertex $x$ to the vertex $gx$, and this action uniquely extends to an isometric action on $S$ since our definition of edges is $G$-equivariant. In particular, the action of $G$ on $S$ is isometric. Clearly, $f$ is $G$-equivariant.
\end{proof}

Let $G$ be a group acting by isometries on a hyperbolic quasi-metric space $Q$ and let $g\in G$. Define that $g$ is loxodromic (resp. satisfying Condition $(\bigstar)$) with respect to the action $G\curvearrowright Q$ in exactly the same manner as for actions of $G$ on Gromov hyperbolic spaces (see Section 3). The following lemma reduces the proof in Section \ref{sec.p}.

\begin{lemma}\label{trivial}
Let $G$ be a group acting by isometries on a hyperbolic quasimetric space $Q$ and let $g\in G$ be a loxodromic element satisfying $(\bigstar)$ with respect to the action $G\curvearrowright Q$. Suppose that $G$ also admits an action on a Gromov hyperbolic space $Q'$ and there is a $G$-equivariant quasi-isometry $f:Q\rightarrow Q'$. Then $g\in G$ is a loxodromicWPD element with respect to the action $G\curvearrowright Q'$.
\end{lemma}

To prove Lemma \ref{trivial}, one checks that $g$ is loxodromic and satisfies $(\bigstar)$ with respect to the action $G\curvearrowright Q'$ and then applies Lemma \ref{star}. We leave the details to the reader.

\section{Proof of Theorem \ref{1.1}}\label{sec.p}

Throughout this section, let $(S,d)$ be a $\delta$-hyperbolic space and let $\partial S$ be the Gromov boundary of $S$. As in Section \ref{sec.h}, pick some point $e\in S$ and define the Gromov product with the aid of $e$. Fix a sufficiently small number $\zeta$ and then define $\rho$ on $\partial S$ so that $\rho$ is a metric and thus induces the topology $\tau$. We will use the notations
$$U_K(x)=\{s\in S\mid (x\cdot s)_e>K\}, ~~~~\partial U_K(x)=\{s\in\partial S\mid (x\cdot S)_e>K\}$$
for $x\in \partial S$ and $K\in \mathbb{R}$. Recall that $B_M(x,r)$ denotes the open ball in a metric space $M$ centered at a point $x\in M$ with radius $r$, that $[u,v]$ denotes a geodesic segment between $u,v\in S$, and that in Remark \ref{reformulate} and Lemma \ref{reformulate1}, we have reformulated Definition \ref{3.5} as the combination of (C$_1$), (C$_2$) and (C$_3$). 

\begin{lemma}\label{7.1}
Let $G$ be a group acting on $S$ by isometries. Then every pair of distinct points of $\partial S$ satisfies (C$_1$).
\end{lemma}

\begin{proof}
Let $x,y$ be two distinct points of $\partial S$. By Lemma \ref{2.6}, there exist $D,R>0$ such that $d(e,[u,v])<D$ for every $u\in U_R(x),v\in  U_R(y)$. By \eqref{4.0}, there exist open subsets $U,V$ of $\partial S$ containing $x,y$ respectively, such that $U\subset \partial U_R(x),V\subset \partial U_R(y)$. We examine (C$_1$) for $U,V$.

Let $p,q$ be two distinct points of $\partial S$. Using Proposition \ref{2.85}, we can find $K>0$ such that $d([a_1,a_2],[b_1,b_2])>2D$ for all $a_1,a_2\in U_K(p)$ and $b\in U_K(q)$. By \eqref{4.0}, there exist open subsets $A,B$ of $\partial S$ containing $p,q$ respectively, such that $A\subset \partial U_K(p),B\subset \partial U_K(q)$.

Suppose that there exists $g\in G$ such that $gA\cap U,gA\cap V,gB\cap U,gB\cap V$ are all non-empty. Let $p'\in gA\cap U,p''\in gA\cap V,q'\in gB\cap U,q''\in gB\cap V$ and let $\{p'_n\}_{n\geqslant 1},\{p''_n\}_{n\geqslant 1},\{q'_n\}_{n\geqslant 1},\{q''_n\}_{n\geqslant 1}$ be sequences in $S$ tending to $p',p'',q',q''$ respectively. Then $\{gp'_n\}_{n\geqslant 1},\{gp''_n\}_{n\geqslant 1},\{gq'_n\}_{n\geqslant 1},\{gq''_n\}_{n\geqslant 1}$ are sequences tending to $gp',gp'',gq',gq''$ respectively. As
$$\min\{(p\cdot p')_e,(p\cdot p'')_e,(q\cdot q')_e,(q\cdot q'')_e\}>K,$$
$$\min\{(x\cdot gp')_e,(x\cdot gq')_e,(y\cdot gp'')_e,(y\cdot gq'')_e\}>R,$$
by Lemma \ref{2.5}, there exists $N>0$ such that
$$\min\{(p\cdot p'_N)_e,(p\cdot p''_N)_e,(q\cdot q'_N)_e,(q\cdot q''_N)_e\}>K,$$
and that
$$\min\{(x\cdot gp'_N)_e,(x\cdot gq'_N)_e,(y\cdot gp''_N)_e,(y\cdot gq''_N)_e\}>R.$$

By our choice of $R$, the geodesics $[gp'_N,gp''_N],[gq'_N,gq''_N]$ intersect $B_S(e,D)$ non-trivially and thus $d([p'_N,p''_N],[q'_N,q''_N])=d([gp'_N,gp''_N],[gq'_N,gq''_N])<2D$. But by our choice of $K$, $d([p'_N,p''_N],[q'_N,q''_N])>2D$, a contradiction.
\end{proof}

\begin{lemma}\label{7.2}
Let $G$ be a group acting on $S$ by isometries. Then every pair of distinct points of $\partial S$ satisfies (C$_2$).
\end{lemma}

\begin{proof}
Let $x,y$ be two distinct points of $\partial S$. By Lemma \ref{2.6}, there exist $D,R>0$ such that $d(e,[u,v])<D$ for every $u\in U_R(x),v\in U_R(y)$. By \eqref{4.0}, there exist open subsets $U,V$ of $\partial S$ containing $x,y$ respectively, such that $U\subset \partial U_R(x),V\subset \partial U_R(y)$. We examine (C$_2$) for $U,V$.

Let $p,q,r$ be three distinct points of $\partial S$. By Proposition \ref{2.115}, there exists $K>0$ such that $d([a,b],[c_1,c_2])>2D$ for every $a\in U_K(p),b\in U_K(q)$ and every $c_1,c_2\in U_K(r)$. By \eqref{4.0}, there exists open subsets $A,B,C$ of $\partial S$ containing $p,q,r$ respectively, such that $A\subset \partial U_K(p),B\subset \partial U_K(q),C\subset \partial U_K(r)$.

Suppose that there exists $g\in G$ such that $gA\cap U,gB\cap V,gC\cap U,gC\cap V$ are all non-empty. Thus, $A\cap g^{-1}U,B\cap g^{-1}V,C\cap g^{-1}U,C\cap g^{-1}V$ are all non-empty. Pick
$$p'\in A\cap g^{-1}U,~~~~q'\in B\cap g^{-1}V, ~~~~r'\in C\cap g^{-1}U,~~~~r''\in C\cap g^{-1}V$$
and let $\{p'_n\}_{n\geqslant 1},\{q'_n\}_{n\geqslant 1},\{r'_n\}_{n\geqslant 1},\{r''_n\}_{n\geqslant 1}$ be sequences in $S$ tending to $p',q',r',r''$ respectively. Then $\{gp'_n\}_{n\geqslant 1},\{gq'_n\}_{n\geqslant 1},\{gr'_n\}_{n\geqslant 1},\{gr''_n\}_{n\geqslant 1}$ are sequences in $S$ tending to $gp',gq',gr',gr''$ respectively. As
$$\min\{(p\cdot p')_e,(q\cdot q')_e,(r\cdot r')_e,(r\cdot r'')_e\}>K,$$
$$\min\{(x\cdot gp')_e,(y\cdot gq')_e,(x\cdot gr')_e,(y\cdot gr'')_e\}>R,$$
by Lemma \ref{2.5}, there exists $N>0$ such that
$$\min\{(p\cdot p'_N)_e,(q\cdot q'_N)_e,(r\cdot r'_N)_e,(r\cdot r''_N)_e\}>K,$$
and that
$$\min\{(x\cdot gp'_N)_e,(y\cdot gq'_N)_e,(x\cdot gr'_N)_e,(y\cdot gr''_N)_e\}>R.$$

By our choice of $R$, the geodesics $[gp'_N,gq'_N],[gr'_N,gr''_N]$ intersect $B_S(e,D)$ non-trivially and hence $d([p'_N,q'_N],[r'_N,r''_N])=d([gp'_N,gq'_N],[gr'_N,gr''_N])<2D$. But by our choice of $K$, $d([p'_N,q'_N],[r'_N,r''_N])>2D$, a contradiction.
\end{proof}

\begin{lemma}\label{7.3}
Let $G$ be a group acting acylindrically on $S$ by isometries. Then every pair of distinct points of $\partial S$ satisfies (C$_3$).
\end{lemma}

\begin{proof}
Let $x,y$ be two distinct points of $\partial S$. By Lemma \ref{2.6}, there exists $R,K>0$ such that $d(e,[u,v])<R$ for every $u\in U_K(x),v\in U_K(y)$.

As the action of $G$ on $S$ is acylindrical, there exists $E>0$ such that for every two points $t,w\in S$ with $d(t,w)\geqslant E$, the number of elements $g\in G$ satisfying both $d(t,gt)\leqslant 189\delta$ and $d(w,gw)\leqslant 189\delta$ is finite.

By Lemmas \ref{2.7} and \ref{2.10}, there exists $F'>K$ such that both of $d(e,u_1)$ and $d(e,v_1)$ are strictly greater than $R+E$ and that $[u_1,v_2]\cap B_S(e,R+E)$ lies inside the $2\delta$-neighborhood of $[u_2,v_2]$ for every $u_1,u_2\in U_{F'}(x)$ and every $v_1,v_2\in U_{F'}(y)$. By Lemma \ref{2.75}, there exists $F>0$ such that $[u_1,u_2]\subset U_{F'}(x)$ for every $u_1,u_2\in U_F(y)$ and that $[v_1,v_2]\subset U_{F'}(y)$ for every $v_1,v_2\in U_F(y)$. Using \eqref{4.0}, we can pick open subsets $U,V$ of $\partial S$ containing $x,y$ respectively, such that $U\subset \partial U_F(x),V\subset \partial U_F(y)$. We examine (C$_3$) for $U,V$.

Suppose, for the contrary, that there exist four distinct points $p,q,r,s$ such that for every four open subsets $A,B,C,D$ of $\partial S$ containing $p,q,r,s$ respectively, we have
$$|\{g\in G\mid gA\cap U,gB\cap U,gC\cap V,gD\cap V\text{ are all non-empty}\}|=\infty.$$
In particular, for $A=B_{\partial S}(p,1),B=B_{\partial S}(q,1),C=B_{\partial S}(r,1),D=B_{\partial S}(s,1)$, there exist $p_1\in A,q_1\in B,r_1\in C,s_1\in D$ and $g_1\in G$ such that $g_1p_1\in U,g_1q_1\in U,g_1r_1\in V,g_1s_1\in V$. For $A=B_{\partial S}(p,1/2),B=B_{\partial S}(q,1/2),C=B_{\partial S}(r,1/2),D=B_{\partial S}(s,1/2)$, since
$$|\{g\in G\mid gA\cap U,gB\cap U,gC\cap V,gD\cap V\text{ are all non-empty}\}|=\infty,$$
there exist $p_2\in A,q_2\in B,r_2\in C,s_2\in D$ and $g_2\in G\backslash\{g_1\}$ such that $g_2p_2\in U,g_2q_2\in U,g_2r_2\in V,g_2s_2\in V$. Continuing in this manner, we see that there exist four sequences $\{p_n\}_{n\geqslant 1},\{q_n\}_{n\geqslant 1},\{r_n\}_{n\geqslant 1},\{s_n\}_{n\geqslant 1}$ of points in $\partial S$ and a sequence $\{g_n\}_{n\geqslant 1}$ of distinct elements in $G$, such that
$$\max\{\rho(p,p_n),\rho(q,q_n),\rho(r,r_n),\rho(s,s_n)\}<\dfrac{1}{n},$$
and that
$$g_np_n\in U,~~~~g_nq_n\in U,~~~~g_nr_n\in V,~~~~g_ns_n\in V,$$
for all $n\geqslant 1$.

By \eqref{4.0}, $\lim_{n\rightarrow\infty}(p\cdot p_n)_e=\lim_{n\rightarrow\infty}(q\cdot q_n)_e=\lim_{n\rightarrow\infty}(r\cdot r_n)_e=\lim_{n\rightarrow\infty}(s\cdot s_n)_e=\infty$. By passing to a subsequence, we may assume that
$$\min\{(p\cdot p_n)_e,(q\cdot q_n)_e,(r\cdot r_n)_e,(s\cdot s_n)_e\}>n,\text{ for all }n.$$

Since $(g_np_n\cdot x)_e>F$ and $(p_n\cdot p)_e>n$, there exists $p'_n\in S$ such that 
$$(g_np'_n\cdot x)_e>F,~~~~(p'_n\cdot p)_e>n,$$
by Lemma \ref{2.5}. Thus, there exist four sequences $\{p'_n\}_{n\geqslant 1},\{q'_n\}_{n\geqslant 1},\{r'_n\}_{n\geqslant 1},\{s'_n\}_{n\geqslant 1}$ of points in $S$ such that
$$\min\{(p'_n\cdot p)_e,(q'_n\cdot q)_e,(r'_n\cdot r)_e,(s'_n\cdot s)_e\}>n,$$
and that
$$\min\{(g_np'_n\cdot x)_e,(g_nq'_n\cdot x)_e,(g_nr'_n\cdot y)_e,(g_ns'_n\cdot y)_e\}>F,$$
for all $n\geqslant 1$.

For each $n$, use the compactness of $[p'_n,q'_n]$ and $[r'_n,s'_n]$ and choose a point $a'_n$ (resp. $b'_n$) in $[p'_n,q'_n]$ (resp. $[r'_n,s'_n]$) such that $d(e,a'_n)=d(e,[p'_n,q'_n])$ (resp. $d(e,b'_n)=d(e,[r'_n,s'_n])$). By Proposition \ref{2.13}, there exists $N>0$ such that if $n\geqslant N$, $[a'_n,b'_n]$ will be in the $92\delta$-neighborhood of $[a'_N,b'_N]$.

By our choice of $F$ and the properties of $\{p'_n\}_{n\geqslant 1},\{q'_n\}_{n\geqslant 1},\{r'_n\}_{n\geqslant 1},\{s'_n\}_{n\geqslant 1}$, we have
$$\min\{(g_na'_n\cdot x),(g_nb'_n\cdot y)\}>F',\text{ for all } n\geqslant 1.$$

By our choice of $F'$, we have the following properties:

\begin{enumerate}[leftmargin=4em]
\item[(P$_1$)] $d(e,[g_Na'_N,g_Nb'_N])<R$;
\item[(P$_2$)] $\min\{d(e,g_Na'_N),d(e,g_Nb'_N)\}>R+E$;
\item[(P$_3$)] $[g_Na'_N,g_Nb'_N]\cap B_S(R+E)$ lies inside the $2\delta$-neighborhood of $[g_na'_n,g_nb'_n]$ for all $n\geqslant 1$.
\end{enumerate}

Pick $c\in [g_Na'_N,g_Nb'_N]$ such that $d(e,c)=d(e,[g_Na'_N,g_Nb'_N])<R$ by (P$_1$) and the compactness of $[g_Na'_N,g_Nb'_N]$. By (P$_2$), there exist $t\in [g_Na'_N,c],w\in [c,g_Nb'_N]$ such that $d(e,t)=d(e,w)=R+E$. As $d(e,c)<R$, we have
$$d(t,w)=d(t,c)+d(c,w)\geqslant 2E.$$
By (P$_3$), $\max\{d(t,[g_na'_n,g_nb'_n]),d(w,[g_na'_n,g_nb'_n])\}\leqslant 2\delta$ for all $n\geqslant N$. Since $g_n$ is an isometry, apply $g^{-1}_n$ and we obtain
$$\max\{d(g^{-1}_nt,[a'_n,b'_n]),d(g^{-1}_nw,[a'_n,b'_n])\}\leqslant 2\delta.$$

For each $n>N$, $[a'_n,b'_n]$ lie inside the $92\delta$-neighborhood of $[a'_N,b'_N]$. Thus, $$\max\{d(g^{-1}_nt,[a'_N,b'_N]),d(g^{-1}_nw,[a'_N,b'_N])\}\leqslant 2\delta+92\delta\leqslant 94\delta.$$
Select a point $z_{t,n}$ (resp. $z_{w,n}$) of $[a'_N,b'_N]$ such that $d(g^{-1}_nt,z_{t,n})\leqslant 94\delta$ (resp. $d(g^{-1}_nw,z_{w,n})$ $\leqslant 94\delta$).

Partition $[a'_N,b'_N]$ into finitely many subpaths such that each of these subpaths has length $<\delta$. Using the Pigeonhole principle, we may assume, after passing to a subsequence, that $z_{t,n}$ stays in a subpath for all $n\geqslant N+1$. Using the Pigeonhole principle once more and passing to a further subsequence, we may further assume that $z_{w,n}$ also stays in a subpath for all $n\geqslant N+1$. Thus, for all $m,n\geqslant N+1$, we have
$$d(g^{-1}_mt,g^{-1}_nt)\leqslant d(g^{-1}_mt,z_{t,m})+d(z_{t,m},z_{t,n})+d(z_{t,n},g^{-1}_nt)<189\delta,$$
$$d(g^{-1}_mw,g^{-1}_nw)\leqslant d(g^{-1}_mw,z_{w,m})+d(z_{w,m},z_{w,n})+d(z_{w,n},g^{-1}_nw)<189\delta.$$

As the $g_n$'s are all distinct, for all $n\geqslant N+1$, we have
$$d(t,g_ng^{-1}_{N+1}t)<189\delta\text{ and }d(w,g_ng^{-1}_{N+1}w)<189\delta.$$

We have found infinitely many elements which move $t,w$ by at most $189\delta$. As $d(t,w)>E$, this contradicts our choice of $E$.
\end{proof}

\begin{proposition}\label{4.6}
Let $G$ be a group acting non-elementarily, acylindrically and isometrically on a Gromov hyperbolic space $S$. Then $G$ is non-virtually-cyclic, has an element with north-south dynamics on $\partial S$ and the action of $G$ on the completely Hausdorff topological space $\partial S$ satisfies (C).
\end{proposition}

Recall that a topological space $M$ is called \textit{completely Hausdorff} if for any two distinct points $u,v\in M$, there are open sets $U,V$ containing $u,v$ respectively, such that $\overline{U}\cap \overline{V}=\emptyset$.

\begin{proof}
By Theorem \ref{acy}, $G$ is not virtually cyclic. By Osin \cite[Theorem 1.1]{acylindrically hyperbolic group}, $G$ contains a loxodromic element $g$ (with respect to the action of $G$ on $S$). By Lemma \ref{no.3.4}, $g$ has north-south dynamics on $\partial S$. As the action of $G$ on $S$ is non-elementary, it is well-known that $|\Lambda(G)|=\infty$ (see \cite{acylindrically hyperbolic group}) and thus $|\partial S|=\infty\geqslant 3$. Let $x,y$ be a pair of distinct points of $M$. Pick open sets $U_1,V_1$ in $M$ containing $x,y$ respectively and satisfying (C$_1$) by Lemma \ref{7.1}, $U_2,V_2$ in $M$ containing $x,y$ respectively and satisfying (C$_2$) by Lemma \ref{7.2}, and $U_3,V_3$ in $M$ containing $x,y$ respectively and satisfying (C$_3$) by Lemma \ref{7.3}. Let $U=U_1\cap U_2\cap U_3,V=V_1\cap V_2\cap V_3$. Then $U,V$ satisfy the (C$_1$), (C$_2$), and (C$_3$) simultaneously. As $x,y$ are arbitrary, Lemma \ref{reformulate1} implies that the action of $G$ on $\partial S$ satisfies (C).
\end{proof}

We now turn to the other direction of Theorem \ref{1.1}.

\begin{proposition}\label{5.7}
Let $G$ be a group acting on a completely Hausdorff topological space $M$ which has at least $3$ points. If there is an element $g\in G$ having north-south dynamics on $M$ and (C$_1$), (C$_2$) and (C$_3$) hold for the fixed points of $g$, then $G$ is either acylindrically hyperbolic or virtually cyclic.
\end{proposition}

\begin{remark}
The existence of a loxodromic element does not follow from the assumption that the action of $G$ satisfies (C). For example, let $G=\mathbb{Z}\times\mathbb{Z}$, let $M=\mathbb{R}^2$ and let $G$ act on $M$ by integral translations, i.e., $(m,n)\cdot (x,y)=(x+m,y+n)$ for all $(m,n)\in G, (x,y)\in M$. As $G$ acts on $M$ properly discontinuously and $M$ is locally compact, it is easy to see that the action of $G$ on $M$ satisfies (C). Nevertheless, no element of $G$ can fix exactly two points of $M$.
\end{remark}

\begin{proof}
Let $x,y$ be the fixed points of $g$. The idea is to construct a specific annulus system on $M$, obtain a Gromov hyperbolic space and then verify that there is a loxodromic WPD element. The construction is illustrated by Figure \ref{fig1}. Since $M$ has at least three points, there is some $z\in M\backslash\{x,y\}$. Pick open sets $U,V$ containing $x,y$ respectively and satisfying (C$_1$), (C$_2$) and (C$_3$). By shrinking $U,V$ if necessary, we may assume that $\overline{U}\cap\overline{V}=\emptyset$ and that $z\not\in\overline{U}\cup\overline{V}$, as $M$ is a completely Hausdorff space. Let
$$A^{-}=\overline{U},~~~~\ A^{+}=\overline{V}.$$
Then $A^{-}$ and $A^{+}$ are two closed sets such that $x\in int A^{-}$, $y\in int A^{+}$, $A^{-}\cap A^{+}=\emptyset, A^{-}\cup A^{+}\neq M$. In Figure \ref{fig1} (I), the white closed half-disc containing $x$ (resp. $y$) is $A^{-}$ (resp. $A^{+}$). The grey shaded region is $M\backslash(A^{-}\cup A^{+})$. Let
$$\mathcal{A}=\{h(\pm A)\mid h\in G\},$$
where $A=(A^{-},A^{+})$. Then $\mathcal{A}$ is a symmetric $G$-invariant annulus system. Define the crossratio $(.,.|.,.)$ and the quasimetric $\rho$ in the same manner as Section \ref{sec.a}.

\begin{figure}
{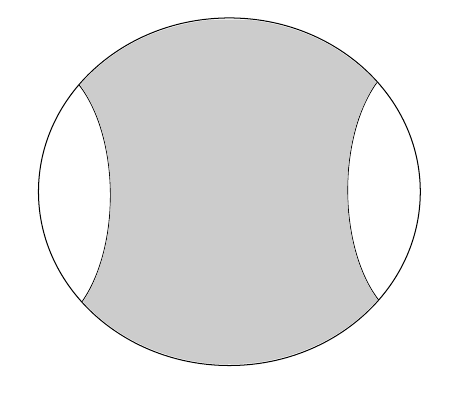}
{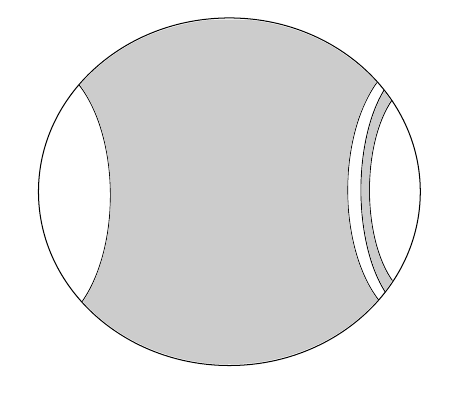}
\caption{(I) The annulus $A$. (II) The image of $A$ under the action of $g^{N}$.
}\label{fig1}
\end{figure}

We proceed to verify that $\mathcal{A}$ satisfies (A1) and (A2). Suppose (A1) does not hold, then there exist four points $p,q,r,s$ such that $p\neq q,r\neq s,(p,q|r,s)=\infty$. By the definition of $(.,.|.,.)$, we see that $p,q,r,s$ are pairwise distinct and, by switching $p$ with $q$ and $r$ with $s$ if necessary, we may assume that there exist infinitely many elements $h\in G$ such that $hp,hq\in U$ and $hr,hs\in V$. Thus, for every open sets $P,Q,R,W$ in $M$ containing $p,q,r,s$ respectively, we have infinitely many elements $h\in G$ such that $hP\cap U,hQ\cap U,hR\cap V,hW\cap V$ are all non-empty and (C$_3$) is violated.

The verification for (A2) is similar. Suppose (A2) does not hold, then there exist four sequences of points $\{p_n\}_{n\geqslant 1},\{q_n\}_{n\geqslant 1},\{r_n\}_{n\geqslant 1},\{s_n\}_{n\geqslant 1}\subset M$ such that for each $n$, $(p_n,q_n|r_n,s_n)>n,(p_n,r_n|q_n,s_n)>n$. We will choose a sequence $\{h_n\}_{n\geqslant 1}$ of distinct elements of $G$ such that $h_nU\cap U,h_nU\cap V,h_nV\cap U,h_nV\cap V$ are non-empty for all $n\geqslant 1$. The verification of (A2) will then be complete since by applying (C$_1$) with $p=x,q=y$, we see that there are only finitely many elements $h\in G$ with $hU\cap U,hU\cap V,hV\cap U,hV\cap V$ all non-empty, a contradiction.

First we choose $h_1$. Since $(p_1,q_1|r_1,s_1)>1,(p_1,r_1|q_1,s_1)>1$, by renaming $p_1,q_1,r_1,s_1$ if necessary, we may assume that there exist $h'_1,h''_1$ such that
$$\{p_1,q_1\}<h'_1A<\{r_1,s_1\},~~~~\{p_1,r_1\}<h''_1A<\{q_1,s_1\}.$$

In other words,
$$p_1\in h'_1U\cap h''_1U, ~~~~q_1\in h'_1U\cap h''_1V, ~~~~r_1\in h'_1V\cap h''_1U, ~~~~s_1\in h'_1V\cap h''_1V.$$

Let $h_1=h'^{-1}_1h''_1$ and we see that $h_1U\cap U,h_1U\cap V,h_1V\cap U,h_1V\cap V$ are all non-empty. 

Suppose that we have chosen $h_1,...,h_{n-1}$. Since $(p_n,q_n|r_n,s_n)>n,(p_n,r_n|q_n,s_n)>n$, there are two elements $h'_n,h''_n\in G$ such that $h'^{-1}_nh''_n$ is not one of $h_1,...,h_{n-1}$ and that (by renaming $p_n,q_n,r_n,s_n$ if necessary)
$$\{p_n,q_n\}<h'_nA<\{r_n,s_n\},~~~~\{p_n,r_n\}<h''_nA<\{q_n,s_n\}.$$

In other words,
$$p_n\in h'_nU\cap h''_nU, ~~~~q_n\in h'_nU\cap h''_nV,~~~~r_n\in h'_nV\cap h''_nU,~~~~ s_n\in h'_nV\cap h''_nV.$$

Let $h_n=h'^{-1}_nh''_n$ and we see that $h_nU\cap U,h_nU\cap V,h_nV\cap U,h_nV\cap V$ are all non-empty and that $h_1,...,h_n$ are all distinct. This finishes the verification of (A2).

Below, we are going to show that $g$ is loxodromic and satisfies $(\bigstar)$ with respect to the action of $G$ on $\Theta_3(M)$. Once this is done, Propositions \ref{hq}, \ref{2}, and Lemma \ref{trivial} will imply that $G$ admits an isometric action on some Gromov hyperbolic space with $g$ being a loxodromic WPD element, and then Theorem \ref{acy} will imply that $G$ is either virtually cyclic or acylindrically hyperbolic, which finishes the proof.

Since $g$ has north-south dynamics on $M$ with fixed points $x,y$, there exists a positive integer $N$ such that $g^{N}(M\backslash intA^{-})\subset intA^{+}$. Figure \ref{fig1} (II) illustrates the dynamics of $g^{N}$ on $M$: $g^{N}$ maps the large grey shaded area onto the small grey shaded band inside of $A^{+}$ and compresses $A^{+}$ into the small white half-disc around $b$ labeled by $g^{N}A^{+}$. From the figure, it is easy to see inequalities (\ref{(2)}), (\ref{(5)}) below. Let $a=(x,y,z)$. To prove that $g$ is loxodromic, it suffices to show that $\rho(a,g^{nN}a)\geqslant n-1$ for all positive integer $n$. Fix a positive integer $n$. Observe that $x,y$ are fixed by $g$, hence $x\in g^{N} (int A^{-}),\ y\in g^{(n-1)N} (int A^{+})$. Consequently,

\begin{equation}\label{(2)}
\{x\}<g^NA,~~~~\{y\}>g^{(n-1)N}A.
\end{equation}

Note that $g$ is a bijection on $M$, thus $g^{N}(M\backslash intA^{-})\subset intA^{+}$ is equivalent to

\begin{equation}\label{(3)}
g^{N}(int A^{-})\cup int A^{+}=M.
\end{equation} 

As a consequence, $A<g^{N}A$. Multiplying both sides of this inequality by $g^{N},\ g^{2N},$ etc, we have the following chain of inequalities:

\begin{equation}\label{(4)}
g^{N}A<g^{2N}A<\cdot\cdot\cdot <g^{(n-1)N}A.
\end{equation}

Since $z\not\in intA^-\cup  int A^+$, equality (\ref{(3)}) also implies

\begin{equation}\label{(5)}
\{z\}<g^{N}A,~~~~A<\{g^{N}z\}.
\end{equation}

The second inequality of (\ref{(5)}) is equivalent to 

\begin{equation}\label{(6)}
g^{(n-1)N}A<\{g^{nN}z\}.
\end{equation}

Combining inequalities (\ref{(2)}), (\ref{(4)}), (\ref{(5)}) and (\ref{(6)}), we obtain
\begin{equation}\label{(imp)}
\{x,z\}<g^{N}A<g^{2N}A<\cdot\cdot\cdot <g^{(n-1)N}A<\{g^{nN}z,y\}.
\end{equation}

Thus, $\rho(a,g^{nN}a)\geqslant (x,z|g^{nN}z,y)\geqslant n-1$ and loxodromicity is proved.

In order to prove $(\bigstar)$, we proceed as follows. Given $\epsilon>0$, let 

\begin{equation}\label{(7)}
L>\epsilon+2,~~~~K=(2L+1)N 
\end{equation}
be integers. By \eqref{(imp)} and (\ref{(7)}), we have $\{x,z\}<A_{1}<A_{2}<\cdot\cdot\cdot <A_{2L}<\{g^{K}z,y\}$, where 

\begin{equation}\label{(8)}
A_{i}=g^{iN}A 
\end{equation}
for all $1\leqslant i\leqslant 2L$. Let us make the following observation.

\begin{lemma}\label{4}
Let $a=(x,y,z)\in \Theta_{3}(M)$. If $w=(w_{1},w_{2},w_{3})\in\Theta_{3}(M)$ and $\rho(a,w)<\epsilon$, then at least two of $w_{1},w_{2},w_{3}$ lie in $A_{L}^{-}$. Similarly, if $\rho(g^{K}a,w)<\epsilon$, then at least two of $w_{1},w_{2},w_{3}$ lie in $A_{L}^{+}$.
\end{lemma}

\begin{proof}
Suppose that $w_{i},w_{j}\not\in A_{L}^{-}$ for some $1\leqslant i<j\leqslant 3$. Since $int A_{L-1}^{+}\cup int A_{L}^{-}=M$ by (\ref{(3)}) and (\ref{(8)}), we have $\{w_{i},w_{j}\}\in int A_{L-1}^{+}$ and consequently $\{x,z\}<A_{1}<A_{2}<\cdot\cdot\cdot <A_{L-1}<\{w_{i},w_{j}\}$. By (\ref{(7)}) and the definition of the quasimetric $\rho$, we have $\rho(a,w)>\epsilon+1$. This proves the first part.

Similarly, suppose $w_{i},w_{j}\not\in A_{L}^{+}$ for some $1\leqslant i<j\leqslant 3$. Again, using (\ref{(3)}) and (\ref{(8)}), we obtain $int A_{L+1}^{-}\cup int A_{L}^{+}=M$. Thus $\{w_{i},w_{j}\}\in int A_{L+1}^{-}$ and consequently $\{w_{i},w_{j}\}<A_{L+1}<A_{L+2}<\cdot\cdot\cdot <A_{2L}<\{g^{K}z,y\}$. As above this implies $\rho(g^{K}a,w)>\epsilon+2$ and proves the second part.
\end{proof}

Now suppose that there is an infinite sequence of distinct elements $\{h_{n}\}_{n\geqslant 1}\subset G$ such that $\rho(a,h_{n}a)<\epsilon,\rho(g^{K}a,h_{n}g^{K}a)<\epsilon$ for all $n$. Since $\rho(a,h_{n}a)<\epsilon$ and $\rho(g^{K}a,h_{n}g^{K}a)<\epsilon$, by Lemma \ref{4}, for every $n$, at least two of $h_{n}x, h_{n}y, h_{n}z$ lie in $A_{L}^{-}$ and at least two of $h_{n}g^{K}x, h_{n}g^{K}y, h_{n}g^{K}z$ lie in $A_{L}^{+}$. There is a subsequence $\{h_{n_k}\}$ and  four points $u_1\neq u_2,v_1\neq v_2$ such that $u_1,u_2\in\{x,y,z\},v_1,v_2\in \{g^Kx,g^Ky,g^Kz\}$ and that $h_{n_k}u_1,h_{n_k}u_2\in A_L^-,h_{n_k}v_1,h_{n_k}v_2\in A_L^+$. In particular, we see that $u_1,u_2,v_1,v_2$ are four distinct points and that $(u_1,u_2|v_1,v_2)=\infty$, which already contradicts the previously proved Axiom (A1). This proves that $g$ satisfies $(\bigstar)$ with respect to the action $G\curvearrowright \Theta_3(M)$ and we are done.
\end{proof}

\begin{corollary}\label{5.x}
Let $G$ be a group which admits an action on a completely Hausdorff space satisfying (C) and contains an element with north-south dynamics. Then $G$ is either acylindrically hyperbolic or virtually cyclic.
\end{corollary}

Theorem \ref{1.1} is now an obvious consequence of Proposition \ref{4.6} and Corollary \ref{5.x}.

By a result of Balasubramanya \cite[Theorem 1.2]{sahana}, an acylindrically hyperbolic group $G$ admits a non-elementary acylindrical and isometric action on one of its Cayley graph $\Gamma$ which is quasi-isometric to a tree $T$. Note that the boundaries $\partial\Gamma$ and $\partial T$ of $\Gamma$ and $T$, respectively, can be naturally identified by a homeomorphism. If, in addition, $G$ is countable, then the construction in \cite{sahana} actually implies that the boundary $\partial T$ of $T$ can be naturally identified, by a homeomorphism, with the \textit{Baire space}, which can be defined as $\mathbb{N}^{\mathbb{N}}$ with the product topology or the set of irrational numbers with the usual topology (see Engelking \cite[Theorem 1.3.13]{set theory} for details).

By Proposition \ref{4.6}, $G$ acts on the Baire space by an action satisfying (C) and has an element with north-south dynamics.

Conversely, if $G$ is a non-virtually-cyclic countable group with an action on the Baire space satisfying (C) and contains an element with north-south dynamics, Corollary \ref{5.x} implies that $G$ is acylindrically hyperbolic. Theorem \ref{7.6} is proved.

Proposition \ref{pre} and Lemma \ref{c-gc} imply that if $G$ is a non-elementary convergence group acting on a compact metrizable topological space $M$, then $G$ is non-virtually-cyclic, has an element with north-south dynamics on $M$ and the action of $G$ on $M$ satisfies (C), thus Corollary \ref{1.2} follows from Theorem \ref{1.1} directly. As mentioned in the introduction, the converse of Corollary \ref{1.2} is not true. In fact, we have the following general statement.

\begin{proposition}\label{counter}
Let $G=\langle X\mid \mathcal{R} \rangle$ be a group generated by $X$ with relations $\mathcal{R}$. If $X$ consists of elements of infinite order and the commutativity graph of $X$ is connected, then any convergence action of $G$ on a compact metrizable topological space is elementary.
\end{proposition}

Here the commutativity graph of $X$ is the undirected graph with vertex set $X$ and edge set consisting of pairs $(x,y)\in X^2$ for every $x,y\in X$ with their commutator $xyx^{-1}y^{-1}$ equal to the identity. The proof of Proposition \ref{counter} is similar to Karlsson-Noskov \cite{KN}, which proves that groups such as $SL_n(\mathbb{Z})$ and Artin braid groups can only have elementary actions on hyperbolic-type bordifications.

\begin{proof}
Suppose that $G$ acts on a compact metrizable topological space $M$ by a convergence action. Let $x$ be an element of $X$. As $x$ has infinite order, it is either parabolic or loxodromic by Remark \ref{pre of cg}. We split our argument into two cases.

\noindent\textbf{Case 1:} $x$ is a parabolic element.

Let $y$ be any element of $X$. As the commutativity graph of $X$ is connected, there exists a path in this graph from $x$ to $y$ labeled by $x=x_1,x_2,...,x_n=y$. Since elements of $X$ have infinite order, everyone of $x_2,...,x_n$ is either parabolic or loxodromic. Let $a\in M$ be the fixed point of $x$.   As $x_1$ commutes with $x_2$, we have
$$x_1x_2a=x_2x_1a=x_2a.$$

In other words, $x_2a$ is a fixed point of $x_1$. Since $x_1$ fixes a unique point, $x_2$ fixes $a$. $x_2$ cannot be a loxodromic element since otherwise the fact that $x_2$ shares the fixed point $a$ with $x_1$ will contradict Remark \ref{pre of cg}. Thus, $x_2$ is a parabolic element fixing $a$. The above argument with $x_2,x_3$ in place of $x_1,x_2$ shows that $x_3$ is also a parabolic element fixing $a$, and then we can apply the argument with $x_3,x_4$ in place of $x_1,x_2$. Continue in this manner and we see that $y=x_n$ is a parabolic element fixing $a$. As $y$ is arbitrary, we conclude that $G$ fixes $a$ and thus is elementary.

\noindent\textbf{Case 2:} $x$ is a loxodromic element.

Let $y$ be any element of $X$. As the commutativity graph of $X$ is connected, there exists a path in this graph from $x$ to $y$ labeled by $x=x_1,x_2,...,x_n=y$. Since elements of $X$ have infinite order, everyone of $x_2,...,x_n$ is either parabolic or loxodromic. Let $a,b\in M$ be the fixed points of $x$.   As $x_1$ commutes with $x_2$, we have
$$x_1x_2a=x_2x_1a=x_2a,~~~~x_1x_2b=x_2x_1b=x_2b.$$

In other words, $x_2a,x_2b$ are two fixed points of $x_1$. Since $x_1$ fixes exactly two points, $x_2$ either permutes $a,b$ or fixes $a,b$ pointwise. If $x_2$ permutes $a,b$, then since $x_2$ is either parabolic or loxodromic, it fixes at least a point $c\in M$ and obviously, $c\neq a,b$. Note that $x^2_2$ has infinite order and fixes three points $a,b,c$, contradicting Remark \ref{pre of cg}.

Thus, $x_2$ fixes $a,b$ pointwise and is a loxodromic element. The above argument with $x_2,x_3$ in place of $x_1,x_2$ shows that $x_3$ is also a loxodromic element fixing $a,b$, and then we can apply the argument with $x_3,x_4$ in place of $x_1,x_2$. Continue in this manner and we see that $y=x_n$ is a loxodromic element fixing $a,b$. As $y$ is arbitrary, we conclude that $G$ fixes $a,b$ and thus is elementary.
\end{proof}

Proposition \ref{counter} implies that various mapping class groups and right-angled Artin groups provide counterexamples for the converse of Corollary \ref{1.2}.

\begin{corollary}
Mapping class groups of closed orientable surfaces with genus $\geqslant 2$ and non-cyclic directly indecomposible right angled Artin groups corresponding to connected graphs are acylindrically hyperbolic groups failing to be non-elementary convergence groups.
\end{corollary}

\begin{proof}
By Osin \cite{acylindrically hyperbolic group}, these groups are acylindrically hyperbolic. For mapping class groups of a closed surface with genus $\geqslant 2$, the commutativity graph corresponding to a generating set due to Wajnryb \cite[Theorem 2]{W} is connected. The fact that a right angled Artin group corresponding to a connected graph has some generating set with connected commutativity graph just follows from the definition. Thus, none of these groups can be a non-elementary convergence group, by Proposition \ref{counter}.
\end{proof}

%
%
%
%


\begin{thebibliography}{9}

\bibitem{sahana}
S. Balasubramanya,
\textit{Acylindrical group actions on quasi-trees}, Alg. Geom. Topology 17.4 (2017), 2145-2176.

\bibitem{BBF}
M. Bestvina, K. Bromberg, K. Fujiwara,
\textit{Bounded cohomology with coefficients in uniformly convex Banach spaces}, Comment. Math. Helv. 91 (2016), no. 2, 203-218.

\bibitem{BF}
M. Bestvina, K. Fujiwara,
\textit{Bounded cohomology of subgroups of mapping class groups},
Geom. Topol. 6 (2202), 69-89

\bibitem{construction}
B. H. Bowditch,
\textit{A topological characterization of hyperbolic groups},
J. Amer. Math. Soc., 11 (1998), no. 3, 643-667.

\bibitem{convergence group}
B. H. Bowditch,
\textit{Convergence groups and configuration spaces},
Geometric Group Theory Down Under: Proceedings of a Special Year in Geometric Group Theory, Canberra, Australia, 1996. Walter de Gruyter, 1999, 23-54.

\bibitem{BH}
M. Bridson, A. Haefliger,
\textit{Metric spaces of non-positive curvature}, 
Grundlehren der Mathematischen Wissenschaften [Fundamental Principles of Mathematical Sciences], vol. 319, Springer-Verlag, Berlin, 1999. 

\bibitem{DGO}
F. Dahmani, V. Guirardel, D. Osin,
\textit{Hyperbolically embedded subgroups and rotating families in groups acting on hyperbolic spaces},
Memoirs Amer. Math. Soc., 245 (2017), no. 1156.

\bibitem{set theory}
R. Engelking,
\textit{Dimension Theory},
Amsterdam: North-Holland Publishing Company, 1978.

\bibitem{Freden}
E. Freden, \textit{Negatively curved groups have the convergence property},
Ann. Acad. Sci. Fenn. Ser. A1, 20 (1995), 333-348.

\bibitem{introduction}
F. W. Gehring, G.J.Martin,
\textit{Discrete quasiconformal groups I},
Proc. London Math. Soc., 55 (1987), 331-358.

\bibitem{GM}
M. Gromov,
\textit{Hyperbolic groups},
Essays in Group Theory, MSRI Series, Vol.8, (S.M. Gersten, ed.), Springer, 1987, 75-263.


\bibitem{no}
M. Hamann,
\textit{Group actions on metric spaces: fixed points and free subgroups},
Abhandlungen aus dem Mathematischen Seminar der Universit\"at Hamburg. Springer-Verlag, Berlin, 2013.

\bibitem{Hull}
M. Hull,
\textit{Small cancellation in acylindrically hyperbolic groups}, 
to appear in Groups, Geom. \& Dynam.

\bibitem{K}
A. Karlsson, 
\textit{Free subgroups of groups with non-trivial Floyd boundary}, 
Comm. Algebra., 31 (2003), 5361-5376.

\bibitem{KN}
A. Karlsson, G. Noskov,
\textit{Some groups having only elementary actions on metric spaces with hyperbolic boundaries}, 
Geometriae Dedicata, 104 (2004), 119-137.

\bibitem{MS}
N. Monod, Y. Shalom,
\textit{Orbit equivalence rigidity and bounded cohomology},
Ann. Math., 164 (2006), 825-878.

\bibitem{acylindrically hyperbolic group}
D. Osin,
\textit{Acylindrically hyperbolic group},
Trans. Amer. Math. Soc., 368 (2016), no. 2, 851-888.

\bibitem{O}
D. Osin,
\textit{Groups acting acylindrically on hyperbolic spaces},
Proc. Int. Cong. of Math. - 2018
Rio de Janeiro, 1 (2018), 915-936.

\bibitem{property}
P. Tukia,
\textit{Convergence groups and Gromov's metric hyperbolic spaces},
New Zealand J. Math., 23 (1994), 157-187.

\bibitem{hyperbolic space}
J. V\"ais\"al\"a,
\textit{Gromov hyperbolic spaces},
Expo. Math.
23
(2005), no. 3, 187-231.

\bibitem{W}
B. Wajnryb,
\textit{A simple presentation for the mapping class group of an orientable surface},
Israel J. Math., 45.2-3 (1983), 157-174.

\bibitem{relative}
A. Yaman,
\textit{A topological characterization of relatively hyperbolic groups},
J. Reine Angew. Math., 566 (2004), 41-89.

\bibitem{floyd}
W. Yang, 
\textit{Growth tightness for groups with contracting elements},
Math. Proc. Cambridge Philos. Soc., 157 (2014), no. 2, 297-319.
\end{thebibliography}
\end{document}